\documentclass[11pt]{article}
\usepackage[maxbibnames=99]{biblatex}
\renewbibmacro{in:}{%
  \ifentrytype{article}{}{\printtext{\bibstring{in}\intitlepunct}}}
\usepackage{amsmath}
\usepackage{amsfonts}
\usepackage{amssymb}
\usepackage{amsthm}
\usepackage{breqn}
\usepackage{setspace}
\usepackage{fullpage}
\usepackage{enumitem}
\usepackage{comment}
\usepackage{hyperref}
\usepackage{bbm}
\usepackage{tikz}
\usepackage{enumitem}
\usepackage{mathtools} 
\usepackage{caption}

\usepackage[utf8]{inputenc}
\usepackage{csquotes}
\def \circ {14pt}
\usepackage[english]{babel}
\usepackage{mathtools}
\usetikzlibrary{patterns,arrows,decorations.pathreplacing}
\newtheorem{lemma}{Lemma}
\newtheorem{theorem}{Theorem} 
 
\newtheorem{definition}{Definition}
\newtheorem{claim}{Claim}

\newtheorem{conjecture}{Conjecture}

\newcommand{\p}{\mathcal{P}}

\newcommand{\floor}[1]{\left\lfloor{#1}\right\rfloor}

\DeclareMathOperator{\ex}{ex}

\usepackage{authblk}
\title{Planar Tur\'an Number of Double Stars}
\author[1,2]{\hspace{1cm}Debarun Ghosh}
\author[1,2]{Ervin Gy\H{o}ri} 
\author[1,2]{Addisu Paulos}
\author[1,2]{Chuanqi Xiao}
\affil[1]{Central European University, Budapest\par
\texttt{ghosh\textunderscore debarun@phd.ceu.edu, xiao\textunderscore chuanqi@outlook.com, addisu\textunderscore 2004@yahoo.com}}
\affil[2]{Alfr\'ed R\'enyi Institute of Mathematics, Budapest \par
\texttt{gyori.ervin@renyi.hu}}
\date{}
\begin{document}
\baselineskip=20 pt
\maketitle
\begin{abstract}
Given a graph $F$, the planar Tur\'an number of $F$, denoted by $\ex_{\p}(n, F)$, is the maximum number of edges in an $n$-vertex $F$-free planar graph. Such an extremal graph problem was initiated by Dowden while determining the sharp upper bound for $\ex_{\p}(n,C_4)$ and $\ex_{\p}(n,C_5)$, where $C_4$ and $C_5$ are cycles of length four and five, respectively.  In this paper, we determine the upper bounds for $\ex_{\p}(n,S_{2,2})$, $\ex_{\p}(n,S_{2,3})$, $\ex_{\p}(n,S_{2,4})$, $\ex_{\p}(n,S_{2,5})$, $\ex_{\p}(n,S_{3,3})$ and $\ex_{\p}(n,S_{3,4})$, where $S_{m,n}$ is a double star with $m$ and $n$ leafs. Moreover, the bounds for $\ex_{\p}(n,S_{2,2})$ and $\ex_{\p}(n,S_{2,3})$ are sharp.   \end{abstract}

\section{Introduction}

One of the famous problems in extremal graph theory is determining the number of edges in an $n$-vertex graph to force a particular graph structure. 
The well-known result of Tur\'an's \cite{turan} gives the maximum number of edges in an $n$-vertex graph containing no complete graph of a given order.  The \emph{Tur\'an number} of a graph $H$, denoted by $\text{ex}(n, H)$, is the maximum number of edges in an $n$-vertex graph that does not contain $H$ as a subgraph. Let $\text{EX}(n,H)$ denote the set of extremal graphs, that is the set of all $n$-vertex, $H$-free graph $G$ such that $e(G)=\text{ex}(n,H)$.  Erd\H{o}s, Stone and Simonovits~\cite{erdos1,erdos2} gave a more generalized result where they determined the asymptotics of $\text{ex}(n,F)$ for all non-bipartite graphs $F$.  They proved that $\text{ex}(n,F)=(1-\frac{1}{\chi(F)-1})\binom{n}{2}+o(n^{2})$, where $\chi(F)$ denotes the chromatic number of $F$.

Over the last decade, a considerable amount of research work has been carried out on Tur\'an type problems. For example, when the host graphs are $K_n$, $k$-uniform hypergraphs, and $k$-partite graphs, see \cite{erdos2,zykov}.  In 2015, Dowden~\cite{dowden} initiated the study of Tur\'an-type problems when the host graph is  planar graph, that is, how many edges a planar graph on $n$ vertices can have without containing a given smaller graph? The planar Tur\'an number of a graph $F$,  denoted by $\text{ex}_{\mathcal{P}}(n, F)$, is the maximum number of edges in a planar graph on $n$ vertices containing no $F$ as a subgraph.  Dowden~\cite{dowden} determined a sharp upper bound for $\ex_{\p}(n, C_4)$ and $\ex_{\p}(n,C_5)$, where $C_4$ and $C_5$ are a cycle of length four and five respectively. Very recently, Ghosh et al.~\cite{ghosh} obtained a sharp upper bound for $\ex_{\p}(n, C_6)$, where $C_6$ is a cycle of length six.  

The analog to Tur\'an's theorem in the case of planar graphs is fairly trivial. Since $K_5$ is not planar, there are only two meaningful cases. For the $K_{3}$, the extremal number of edges is $2n-4$, and the extremal graph is $K_{2,n-2}$ (since all faces have size four when drawn in the plane).  Note that there exist planar triangulations not containing $K_4$ (e.g., take a cycle of length $n - 2$ and then add two new vertices which are adjacent to all those in the cycle). Thus, the extremal number in the case of $K_4$ is $3n - 6$. The planar Tur\'an number when the forbidden subgraph is a star is also fairly trivial.  The authors in \cite{lan_Hfreeplanar} proved that $\ex_{\mathcal{P}} (n, H) = 3n - 6$ for all $H$ with $n > |H|+2$ and either $\chi(H) = 4$ or $\chi(H) = 3$ and $\Delta(H) > 7$. They also completely determined $\ex_{\p} (n, H)$ when $H$ is a wheel or a star, and the case when $H$ is a $(t, r)$-fan, that is, $H$ is isomorphic to $K_1 + tK_{r-1}$, where $t > 2$  and $r > 3$ are integers. The next most natural type of graph to investigate is perhaps a path. For extremal planar Tur\'an number for paths of length $\{6, 7, 8, 9, 10, 11\}$, we refer the reader to \cite{lan2} and \cite{lan_Hfreeplanar}.  Planar Tur\'an number of other graphs, for instance, wheels and Theta graphs also be considered. We refer the reader to \cite{ghosh2, lan2, lan} for detailed results. The next natural extension of the topic is considering double stars as the forbidden graph.

\begin{definition}
An \emph{($m$,$n$)-double star}, denoted by $S_{m,n}$, is the graph  obtained by taking an edge, say $xy$, and joining one of its end vertices, say $x$, with $m$ vertices and  the other end vertex, $y$, with $n$ vertices which are different from the $m$ vertices.  The edge $xy$ is called the \emph{backbone} of the double star. The vertices adjacent to an end vertex of the backbone are called the \emph{leaf-sets} of the double star. Figure \ref{double} shows an $m$-$n$ double star such that the backbone is $xy$ and the leaf-sets are $\{x_1,x_2,\dots,x_m\}$ and $\{y_1,y_2,\dots,y_n\}$, respectively. 
\end{definition}

\begin{figure}[ht]
\centering
\begin{tikzpicture}[scale=0.1253]
\draw[fill=black](-10,0)circle(20pt);
\draw[fill=black](10,0)circle(20pt);
\draw[fill=black](20,10)circle(15pt);
\draw[fill=black](20,6)circle(15pt);
\draw[fill=black](20,2)circle(15pt);
\draw[fill=black](20,-3)circle(5pt);
\draw[fill=black](20,-4)circle(5pt);
\draw[fill=black](20,-5)circle(5pt);
\draw[fill=black](20,-8)circle(15pt);
\draw[fill=black](-20,10)circle(15pt);
\draw[fill=black](-20,6)circle(15pt);
\draw[fill=black](-20,2)circle(15pt);
\draw[fill=black](-20,-3)circle(5pt);
\draw[fill=black](-20,-4)circle(5pt);
\draw[fill=black](-20,-5)circle(5pt);
\draw[fill=black](-20,-8)circle(15pt);
\draw[black, ultra thick](-10,0)--(10,0)(10,0)--(20,10)(10,0)--(20,6)(10,0)--(20,2)(10,0)--(20,-8)(-10,0)--(-20,10)(-10,0)--(-20,6)(-10,0)--(-20,2)(-10,0)--(-20,-8);
\node at (-10,-2) {$x$};
\node at (10,-2) {$y$};
\node at (22,10) {$y_1$};
\node at (22,6) {$y_2$};
\node at (22,2) {$y_3$};
\node at (22,-8) {$y_n$};
\node at (-23,10) {$x_1$};
\node at (-23,6) {$x_2$};
\node at (-23,2) {$x_3$};
\node at (-23,-8) {$x_m$};
\end{tikzpicture}
\caption{A double star $S_{m,n}$ with $m$ and $n$ leafs.}
\label{double}
\end{figure}
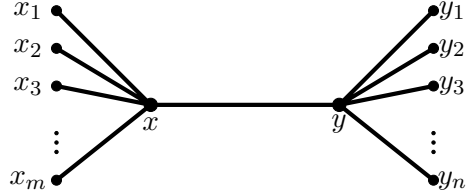

In this paper, we address the upper bounds of planar Tur\'an number of  the following double stars: $S_{2,2}$, $S_{2,3}$, $S_{2,4}$, $S_{2,5}$, $S_{3,3}$ and $S_{3,4}$.  Moreover, the bound for $\ex_{\p}(n,S_{2,2})$ and $\ex_{\p}(n,S_{2,3})$ is sharp. The bound for $\ex_{\p}(n,S_{3,3})$ is sharp up-to a linear term.  Before proceeding to our results, we mention notations and terminologies to be used in next. 

All the graphs we consider in this paper are simple and finite. Let $G$ be a graph. We denote the vertex and edge set of $G$ by $V(G)$ and $E(G)$, respectively. Let $e(G)$ and $v(G)$ denote the number of edges and vertices, respectively.  We denote the degree of a vertex $v$ by $d(v)$, the minimum degree in graph $G$ by $\delta(G)$ and the maximum degree in graph $G$ by $\Delta(G)$. The subgraph induced by $S\subseteq V(G)$, is denoted by $G[S]$. Moreover, $N(v)$ denotes the set of vertices in $G$ adjacent to $v$. Let $A$ and $B$ be disjoint subsets of $V(G)$. Let $e(A,B)$ denote the number of edges in $G$, that joins a vertex in $A$ and a vertex in $B$. An $m$-$n$ edge is an edge such that the end vertices of the edge are with degree $m$ and $n$.  The join $G=G_1+G_2$ of graphs $G_1$ and $G_2$ with disjoint vertex sets $V_1$ and $V_2$ and edge sets $X_1$ and $X_2$ is the graph union $G_1\cup G_2$ together with all the edges joining $V_1$ and $V_2$. 

The following theorem summarizes the main results:
\begin{theorem}\label{doublestarsmain}
Estimates on the planar Tur\'an number of double stars $S_{m,n}$ for given values of $\{m,n\}$ are as follows:
\begin{enumerate}[label=(\roman*)]
    \item For any $n\geq 16$, $\ex_{\p}(n,S_{2,2})=  2n-4$.
    \item For any $n\geq 1$, $\ex_{\mathcal{P}}(n,S_{2,3})=2n.$
    \item For any $n\geq 1$, $\frac{15}{7}n \leq \ex_{\mathcal{P}}(n,S_{2,4})\leq\frac{8}{3}n.$
    \item For $n\geq 1$, $\frac{5}{2}n \leq \ex_{\mathcal{P}}(n,S_{2,5})\leq\frac{20}{7}n.$
    \item For $n\geq 3$, $\frac{5}{2}n-5\leq\ex_{\mathcal{P}}(n,S_{3,3})\leq\frac{5}{2}n-2.$
    \item For $n\geq 1$, $\frac{9}{4}n\leq \ex_{\mathcal{P}}(n,S_{3,4})\leq\frac{20}{7}n.$
\end{enumerate}
\end{theorem}

\section{Planar Tur\'an number of \texorpdfstring{$S_{2,2}$}{S2,2}}
We start by proving the following weaker bounds:    
\begin{lemma}\label{mm3}
Let $G$ be an $S_{2,2}$-free plane graph on $n$ ($n\neq 5$) vertices, then $e(G)\leq 2n-2$.
\end{lemma}
\begin{proof}
Suppose that $G$ contains $6$ vertices. There are only two $6$-vertex maximal planar graphs $M_1$ and $M_2$ as shown in Figure~\ref{bbs}. It can be checked that $M_1^-$ and $M_2^-$ both contain an $S_{2,2}$. Thus, $e(G)\leq 10=2n-2$, when $n=6$. 
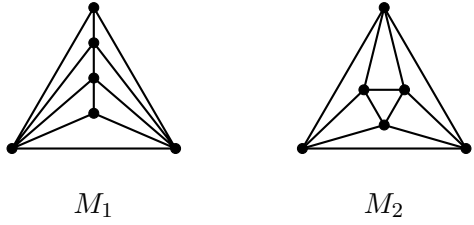
\begin{figure}[ht]
\centering
\begin{tikzpicture}[scale=0.125]
\draw[thick](0,10)--(8.7,-5)--(-8.7,-5)--(0,10)(0,10)--(0,-1.25)(-8.7,-5)--(0,6.25)--(8.7,-5)(-8.7,-5)--(0,2.5)--(8.7,-5)(-8.7,-5)--(0,-1.25)--(8.7,-5);
\draw[fill=black](0,10)circle(15pt);
\draw[fill=black](0,6.25)circle(15pt);
\draw[fill=black](0,2.5)circle(15pt);
\draw[fill=black](0,-1.25)circle(15pt);
\draw[fill=black](8.7,-5)circle(15pt);
\draw[fill=black](-8.7,-5)circle(15pt);
\node at (0,-11) {$M_1$};
\end{tikzpicture}\qquad\qquad
\begin{tikzpicture}[scale=0.125]
\draw[thick](0,10)--(8.7,-5)--(-8.7,-5)--(0,10);
\draw[thick](0,-2.5)--(2.16,1.25)--(-2.16,1.25)--(0,-2.5)(-8.7,-5)--(-2.16,1.25)--(0,10)(8.7,-5)--(2.16,1.25)--(0,10)(-8.7,-5)--(0,-2.5)--(8.7,-5);
\draw[fill=black](0,-2.5)circle(15pt);
\draw[fill=black](2.16,1.25)circle(15pt);
\draw[fill=black](-2.16,1.25)circle(15pt);
\draw[fill=black](0,10)circle(15pt);
\draw[fill=black](8.7,-5)circle(15pt);
\draw[fill=black](-8.7,-5)circle(15pt);
\node at (0,-11) {$M_2$};
\end{tikzpicture}
\caption{The two $6$-vertex maximal planar graphs.}
\label{bbs}
\end{figure}
Now let $G$ be an $S_{2,2}$-free planar graph on $7$ vertices. If $G$ contains a vertex of degree at most $2$, we are done by induction. Moreover, there is no vertex of degree at least $5$ in $G$. Suppose there is a vertex $x\in V(G)$ such that $d(x)=k$, where $k\geq 5$. In this case, each vertex in $N(x)$ must be of degree at most $2$. Otherwise, it is easy to find an $S_{2,2}$ in $G$, which is a contradiction. 

Assume that $\Delta(G)\leq4$. Let the number of vertices in $G$ with degree at most $3$ be $k$. Hence, $G$ contains at least $n-k$ vertices of degree $4$, which implies $$e(G)\leq\frac{4(n-k)+3k}{2}=2n-\frac{k}{2}.$$ If there are at least $4$ vertices of degree at most $3$, then $e(G)\leq2n-\frac{4}{2}=2n-2$.  Let $v$ be a degree $4$ vertex in $G$. If each vertex in $N(v)$ is of degree at most $3$, then $e(G)\leq 2n-2$. So, there is a vertex $u\in N(v)$, such that $uv$ is a $4$-$4$ edge in $G$. Since $G$ is an $S_{2,2}$-free plane graph, $uv$ must be contained in $3$ triangles. Let $N(u)\cap N(v)=\{x_1,x_2,x_3\}$ and $S=V(G)\backslash\{u,v,x_1,x_2,x_3\}$. Observe that no vertex in $S$ is adjacent to a vertex in $\{x_1,x_2,x_3\}$.  Deleting the vertices $\{u,v,x_1,x_2,x_3\}$.  We deleted at most $3\cdot 5-6=9$ edges.  Therefore, $e(G)=e(G-\{u,v,x_1,x_2,x_3\})+9\leq 2(n-5)-2+9\leq 2n-2$.  Hence, we are done by induction.
\end{proof}

\begin{lemma}\label{mm12}
Let $G$ be an $S_{2,2}$-free plane graph on $n$ ($n\geq 8$) vertices. If there is a vertex with degree at least $5$, then $e(G)\leq 2n-4$.
\end{lemma}
\begin{proof}
Let $x \in V (G)$ such that $d(x) = k\geq 5$. Let $N(x) = \{x_1, x_2, x_3,\dots , x_k\}$, and $S$ be the set of vertices in $V(G)\backslash N(x)$. Each vertex in $N(x)$ is adjacent to at most one other vertex in $N(x)$. Otherwise, it is easy to show that $G$ contains an $S_{2,2}$. Similarly, for an edge $x_ix_j$, where $x_i,x_j\in N(x)$, there is no vertex in $S$ which is adjacent to either $x_i$ or $x_j$. Thus, the number of edges joining a vertex in $N(x)$ and a vertex in $S$ is at most $k$.  If $|S|\neq 5$, using Lemma \ref{mm3}, $e(G[S]) \leq 2(n-k-1)-2.$ Therefore, $e(G) \leq 2(n-k-1)-2+k+k =2n-4$, and we are done.  Let $|S|= 5$. Let the graph induced by $S$ be a $K_5^-$.  Since $G$ is an $S_{2,2}$-free plane graph, no vertex in $N(x)$ is adjacent to any vertex in $S$.  This implies that $e(G[S]) = 9$ and $e(G[N(x)])\leq \frac{k}{2}$. Hence, $e(G)\leq k + \frac{k}{2} + 9$. Since $n = k + 6$, we have $e(G) \leq (n-6) + \frac{n-6}{2}+ 9 = \frac{3n}{2}\leq 2n-4$ for $n\geq 8$.
\end{proof}

\begin{proof}[Proof of Theorem \ref{doublestarsmain}(i)]
The lower bound is attained by considering the graph $K_{2,n-2}$, which is $S_{2,2}$-free and contains $2n-4$ edges. From Lemma~\ref{mm12}, we may assume that $\Delta(G)\leq 4$. Let $k$ be the number of vertices in $G$ whose degree is at most $3$. Then $e(G)\leq \frac{4(n-k)+3k}{2}=2n-\frac{k}{2}$.  If $k\geq 8$, then we are done. We may assume that the number of degree $4$ vertices in $G$ is at least $9$, since $n\geq 16$. We start by proving the following claims:

\begin{claim}\label{mms2}
There is no degree $4$ vertex in $G$ such that all its neighbors are of degree at most $3$.
\end{claim}
\begin{proof}
Suppose not.  Let $x$ be a degree $4$ vertex in $G$, such that $d(y)\leq 3$ for all $y\in N(x)$.  It is easy to check that for each $y\in N(x)$, if $y$ is adjacent to any vertex in $V(G)\backslash (\{x\}\cup N(x))$, then $d(y)=2$. Otherwise, $G$ contains an $S_{2,2}$.  Therefore, using Lemma~\ref{mm3}, $e(G)=e(G[S])+4+4\leq 2(n-5)-2+8=2n-4$.    Assume that for each $y\in N(x)$, $y$ is not adjacent to any vertex in $V(G)\backslash (\{x\}\cup N(x))$. Then $e(G[x\cup N(x)])\leq 8$. Therefore, using Lemma~\ref{mm3}, $e(G)=e(G[S])+8\leq 2(n-5)-2+8=2n-4.$ 
\end{proof}

\begin{claim}
The number of $4$-$4$ edges in a matching in $G$ is at least $3$.
\end{claim}
\begin{proof}
Suppose not. Let the number of $4$-$4$ edges in a matching in $G$ be $2$.  Denote the $4$-$4$ edges in the matching by $uv$ and $xy$. Each of the edges is contained in $3$ triangles. Let $N(u)\cap N(v)=\{u_1,u_2,u_3\}$ and $N(x)\cap N(y)=\{x_1,x_2,x_3\}$. Since $G$ is an $S_{2,2}$-free plane graph, no vertex in $\{u_1,u_2,u_3\}$ is adjacent to a vertex in $V(G)\backslash \{x,y,x_1,x_2,x_3\}$ and no vertex in $\{x_1,x_2,x_3\}$ is adjacent to a vertex in $V(G)\backslash \{u,v,u_1,u_2,u_3\}$. Moreover, at least two vertices in $\{x_1,x_2,x_3\}$ and in $\{u_1,u_2,u_3\}$ are of degree at most three. The number of degree $4$ vertices in $G$ is at least $9$. Thus, there is a degree $4$ vertex, say $z$,  in $V(G)\backslash \{x,y,u,v,x_1,x_2,x_3,u_1,u_2,u_3\}$ such that $d(t)\leq 3$ for each $t\in N(z)$. This contradicts Claim~\ref{mms2}. A similar argument can be given, if we assume that the number of $4$-$4$ edges in a matching in $G$ is $1$.
\end{proof}

From now on, suppose that the number of $4$-$4$ edges in a matching in $G$ is at least $3$. We distinguish the following two cases:

\textbf{Case $1$: The number of $4$-$4$ edges in a matching in $G$ is at least $4$.}
Denote the $4$-$4$ edges in the matching in $G$ by $x_1x_2$, $x_3x_4$, $x_5x_6$ and $x_7x_8$, respectively. Recall that each edge is contained in $3$ triangles. Moreover, at least two vertices in $N(x_i)\cap N(x_{i+1})$ are of degree at most $3$, for each $i\in\{1,3,5,7\}$. Thus, we have at least $8$ vertices in $G$ whose degree is at most $3$. Hence, we are done in this case.

\textbf{Case $2$: The number of $4$-$4$ edges in a  matching in $G$ is $3$.}
Denote the $4$-$4$ edges in the matching in $G$ by $x_1x_2, x_3x_4$ and $x_5x_6$, respectively. At least two vertices in $N(x_i)\cap N(x_{i+1})$ are of degree at most $3$, for $i\in\{1,3,5\}$. This implies that $G$ contains at least $6$ vertices whose degree is at most $3$. Moreover, if a vertex in $N(x_i)\cap N(x_{i+1})$ is of degree at most $2$, for some $i\in\{1,3,5\}$, then the remaining two vertices are of degree at most $3$.  Observe that in this case, $e(G)\leq\frac{4(n-7)+3\cdot6+2}{2}=2n-4$ and we are done.

So, we assume exactly two vertices in $N(x_i)\cap N(x_{i+1})$ are of degree $3$, for each $i\in\{1,3,5\}$. In this case, the remaining vertex in $N(x_i)\cap N(x_{i+1})$ is of degree $4$. Thus, the vertices $\{x_i,x_{i+1}\}\cup \bigg(N(x_i)\cap N(x_{i+1})\bigg)$, for each $i\in\{1,3,5\}$, induce a $K_5^-$, and it is a component in $G$.  If $n=16$, then there is an isolated vertex. In this case, $e(G)=27<2n-4.$ If $n>17$, it is easy to find $2$ more vertices of degree at most $3$ since the number of $4$-$4$ edges in a matching in $G$ is $3$. Thus, there are at least $8$ degree $3$ vertices in $G$.  This completes the proof of Theorem~\ref{doublestarsmain}(i).  
\end{proof}

\section{Planar Tur\'an number of \texorpdfstring{$S_{2,3}$}{S2,3}}
\begin{proof}[Proof of theorem \ref{doublestarsmain}(ii)]
Let $G$ be an $S_{2,3}$-free plane graph on $n$ vertices.  Since the graph $S_{2,3}$ contains $7$ vertices, a maximal planar graph with $n\leq6$  vertices does not contain an $S_{2,4}$.  Thus, the lower bound is attained by considering disjoint copies of the maximum planar graphs on $6$ vertices, i.e., $M_1$ or $M_2$ (see Figure \ref{bbs}.  If the maximum degree in $G$ is at most $4$, then $e(G)\leq 2n$. Now we separate the rest of the proof into $2$ cases:

\textbf{Case $1$: There exists a vertex $v\in V(G)$, such that $d(v)\geq 6$.}
It is easy to check that for each $u\in N(v)$, $d(u)\leq 2$, otherwise, we find a copy of $S_{2,3}$ in $G$. Delete a vertex $u\in N(v)$, then the number of deleted edges is at most $2$. By the induction hypothesis, we get $e(G-\{u\})\leq 2(n-1)$. Hence, $e(G)=e(G-\{u\})+d(u)\leq 2(n-1)+2=2n$.

\textbf{Case $2$: There exists a vertex $v\in V(G)$, such that $d(v)=5$.}
It is easy to check that for each $u\in N(v)$, if $u$ is adjacent to any vertex in $V(G)\backslash (\{v\}\cup N(v))$, then $d(u)=2$. Otherwise, $G$ contains an $S_{2,3}$.  As in Case $1$, we are done by induction.  Assume that for each $u\in N(v)$, $u$ is not adjacent to any vertex in $V(G)\backslash (\{v\}\cup N(v))$. Then $e(G[v\cup N(v)])\leq 3\cdot 6-6=12$. By the induction hypothesis, $e(G-\{v\cup N(v)\})\leq 2(n-6)$. Therefore, $e(G)=e(G-\{v\cup N(v)\})+e(G[v\cup N(v)])\leq 2n$.
\end{proof}

\section{Planar Tur\'an number of \texorpdfstring{$S_{2,4}$}{S2,4}}
Let $G$ be an $S_{2,4}$-free plane graph on $n$ vertices.  Since $S_{2,4}$ contains $8$ vertices, a maximal planar graph with $n\leq 7$ vertices, does not contain an $S_{2,4}$.  Let $7|n$. Consider the plane graph consisting of $\frac{n}{7}$ disjoint copies of maximal planar graphs on $7$ vertices.  This graph does not contain an $S_{2,4}$. Hence, $\text{ex}_{\mathcal{P}}(n,S_{2,4})\geq\frac{15}{7}n$.

\begin{claim}
Let $G$ be an $S_{2,4}$ on $n$ $(1\leq n\leq 18)$ vertices.  The number of edges in $G$ is at most $\frac{8}{3}n$.
\end{claim}
\begin{proof}
Recall that, an $n$-vertex maximal planar graph contains $3n-6$ edges. Since $3n-6\leq \frac{8}{3}n$ for $n\leq 18$,  $e(G)\leq \frac{8}{3}n$ holds for all $n$, $1\leq n\leq 18$.
\end{proof}

\begin{lemma}\label{deg7_S_2,4}
If $G$ contains a vertex of degree greater than or equal to $7$, then $e(G)\leq \frac{8}{3}n.$
\end{lemma}
\begin{proof}
Let $v\in V(G)$, such that $d(v)\geq 7$.  It is easy to check that for each $u\in N(v)$, $d(u)\leq 2$, otherwise, we find a copy of $S_{2,4}$ in $G$. Delete a vertex $u\in N(v)$, then the number of deleted edges is at most $2$. By the induction hypothesis, we get $e(G-\{u\})\leq \frac{8}{3}(n-1)$. Hence, $e(G)=e(G-\{u\})+d(u)\leq \frac{8}{3}(n-1)+2\leq\frac{8}{3}n$.
\end{proof}

\begin{lemma}\label{deg6_S_2,4}
If $G$ contains a vertex of degree $6$, then $e(G)\leq \frac{8}{3}n.$
\end{lemma}
\begin{proof}
Let $v\in V(G)$, such that $d(v)=6$ and let $H=N(v)\cup \{v\}$.  It is easy to check that for each $u\in N(v)$, if $u$ is adjacent to any vertex in $V(G)\backslash H$, then $d(u)=2$. Otherwise, $G$ contains an $S_{2,4}$.  We are done by induction in this case.  Assume that for each $u\in N(v)$, $u$ is not adjacent to any vertex in $V(G)\backslash H$. Then $e(G[H])\leq 3\cdot 7-6=15$. Hence, $e(G)=e(G-H)+e(G[H])\leq \frac{8}{3}(n-7)+15\leq\frac{8}{3}n$. We are done by induction. 
\end{proof}

\begin{proof}[Proof of Theorem \ref{doublestarsmain}(iii)]
If $G$ contains a vertex of degree at least $6$, we are done by Lemmas \ref{deg7_S_2,4} and \ref{deg6_S_2,4}. Hence, we can assume that $\Delta(G)\leq 5$.
\begin{figure}[ht]
\centering
    \begin{tikzpicture}[scale=0.1]
            \coordinate (x) at (-12,0);
            \coordinate (y) at (12,0);
            \coordinate (a) at (0,27);
            \coordinate (b) at (0,18);
            \coordinate (c) at (0,9);
            \coordinate (d) at (0,-9);
            \draw[thick] (x) -- (a) -- (y) -- (x);
            \draw[thick] (x) -- (b) -- (y);
            \draw[thick] (x) -- (c) -- (y);
            \draw[thick] (x) -- (d) -- (y);
            \draw[fill=black] (x) circle(\circ)  node[label=below:$x$] {};
            \draw[fill=black] (y) circle(\circ)  node[label=below:$y$] {};
            \draw[fill=black] (a) circle(\circ)  node[label=below:$a$] {};
            \draw[fill=black] (b) circle(\circ)  node[label=below:$b$] {};
            \draw[fill=black] (c) circle(\circ)  node[label=below:$c$] {};
            \draw[fill=black] (d) circle(\circ)  node[label=above:$d$] {};
    \end{tikzpicture}
    \caption{The graph $G$ has a $5-5$ edge $xy$, with $4$ triangles sitting on the edge $xy$.}
    \label{fig_s_2,4_5,5,4triangles}
\end{figure}
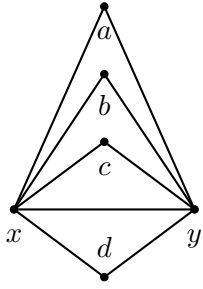
\begin{figure}
\centering
    \begin{tikzpicture}[scale=0.1]
            \coordinate (x) at (-12,0);
            \coordinate (y) at (12,0);
            \coordinate (a) at (0,18);
            \coordinate (b) at (0,9);
            \coordinate (c) at (0,-9);
            \coordinate (d) at (-21,0);
            \coordinate (e) at (21,0);
            \draw[thick] (x) -- (a) -- (y) -- (x);
            \draw[thick] (x) -- (b) -- (y);
            \draw[thick] (x) -- (c) -- (y);
            \draw[thick] (x) -- (d) -- (y);
            \draw[thick] (x) -- (d);
            \draw[thick] (e) -- (y);
            \draw[fill=black] (x) circle(\circ)  node[label=below:$x$] {};
            \draw[fill=black] (y) circle(\circ)  node[label=below:$y$] {};
            \draw[fill=black] (a) circle(\circ)  node[label=below:$a$] {};
            \draw[fill=black] (b) circle(\circ)  node[label=below:$b$] {};
            \draw[fill=black] (c) circle(\circ)  node[label=below:$c$] {};
            \draw[fill=black] (d) circle(\circ)  node[label=below:$d$] {};
            \draw[fill=black] (e) circle(\circ)  node[label=below:$e$] {};
            \node at (0,-21){(i)};
    \end{tikzpicture}
    \qquad
    \begin{tikzpicture}[scale=0.1]
            \coordinate (x) at (-12,0);
            \coordinate (y) at (12,0);
            \coordinate (a) at (0,18);
            \coordinate (b) at (0,9);
            \coordinate (c) at (0,-9);
            \coordinate (d) at (-21,0);
            \coordinate (e) at (21,0);
            \draw[thick] (x) -- (a) -- (y) -- (x);
            \draw[thick] (x) -- (b) -- (y);
            \draw[thick] (x) -- (c) -- (y);
            \draw[thick] (x) -- (d) -- (y);
            \draw[thick] (x) -- (d);
            \draw[thick] (e) -- (y);
            \draw[thick] (d) -- (a);
            \draw[fill=black] (x) circle(\circ)  node[label=below:$x$] {};
            \draw[fill=black] (y) circle(\circ)  node[label=below:$y$] {};
            \draw[fill=black] (a) circle(\circ)  node[label=below:$a$] {};
            \draw[fill=black] (b) circle(\circ)  node[label=below:$b$] {};
            \draw[fill=black] (c) circle(\circ)  node[label=below:$c$] {};
            \draw[fill=black] (d) circle(\circ)  node[label=below:$d$] {};
            \draw[fill=black] (e) circle(\circ)  node[label=below:$e$] {};
            \node at (0,-21){(ii)};
    \end{tikzpicture}
    \qquad
    \begin{tikzpicture}[scale=0.1]
            \coordinate (x) at (-12,0);
            \coordinate (y) at (12,0);
            \coordinate (a) at (0,18);
            \coordinate (b) at (0,9);
            \coordinate (c) at (0,-9);
            \coordinate (d) at (-21,0);
            \coordinate (e) at (21,0);
            \draw[thick] (x) -- (a) -- (y) -- (x);
            \draw[thick] (x) -- (b) -- (y);
            \draw[thick] (x) -- (c) -- (y);
            \draw[thick] (x) -- (d) -- (y);
            \draw[thick] (x) -- (d);
            \draw[thick] (e) -- (y);
            \draw[thick] (d) -- (a);
            \draw[thick] (e) -- (a);
            \draw[fill=black] (x) circle(\circ)  node[label=below:$x$] {};
            \draw[fill=black] (y) circle(\circ)  node[label=below:$y$] {};
            \draw[fill=black] (a) circle(\circ)  node[label=below:$a$] {};
            \draw[fill=black] (b) circle(\circ)  node[label=below:$b$] {};
            \draw[fill=black] (c) circle(\circ)  node[label=below:$c$] {};
            \draw[fill=black] (d) circle(\circ)  node[label=below:$d$] {};
            \draw[fill=black] (e) circle(\circ)  node[label=below:$e$] {};
            \node at (0,-21){(iii)};
    \end{tikzpicture}
    \qquad
    \begin{tikzpicture}[scale=0.1]
            \coordinate (x) at (-12,0);
            \coordinate (y) at (12,0);
            \coordinate (a) at (0,18);
            \coordinate (b) at (0,9);
            \coordinate (c) at (0,-9);
            \coordinate (d) at (-21,0);
            \coordinate (e) at (21,0);
            \draw[thick] (x) -- (a) -- (y) -- (x);
            \draw[thick] (x) -- (b) -- (y);
            \draw[thick] (x) -- (c) -- (y);
            \draw[thick] (x) -- (d) -- (y);
            \draw[thick] (x) -- (d);
            \draw[thick] (e) -- (y);
            \draw[thick] (d) -- (a);
            \draw[thick] (e) -- (c);
            \draw[fill=black] (x) circle(\circ)  node[label=below:$x$] {};
            \draw[fill=black] (y) circle(\circ)  node[label=above:$y$] {};
            \draw[fill=black] (a) circle(\circ)  node[label=below:$a$] {};
            \draw[fill=black] (b) circle(\circ)  node[label=below:$b$] {};
            \draw[fill=black] (c) circle(\circ)  node[label=below:$c$] {};
            \draw[fill=black] (d) circle(\circ)  node[label=below:$d$] {};
            \draw[fill=black] (e) circle(\circ)  node[label=below:$e$] {};
            \node at (0,-21){(iv)};
    \end{tikzpicture}
    \qquad
    \begin{tikzpicture}[scale=0.1]
            \coordinate (x) at (-12,0);
            \coordinate (y) at (12,0);
            \coordinate (a) at (0,18);
            \coordinate (b) at (0,9);
            \coordinate (c) at (0,-9);
            \coordinate (d) at (-21,0);
            \coordinate (e) at (21,0);
            \draw[thick] (x) -- (a) -- (y) -- (x);
            \draw[thick] (x) -- (b) -- (y);
            \draw[thick] (x) -- (c) -- (y);
            \draw[thick] (x) -- (d) -- (y);
            \draw[thick] (x) -- (d);
            \draw[thick] (e) -- (y);
            \draw[thick] (d) -- (a);
            \draw[thick] (d) -- (c);
            \draw[fill=black] (x) circle(\circ)  node[label=above:$x$] {};
            \draw[fill=black] (y) circle(\circ)  node[label=above:$y$] {};
            \draw[fill=black] (a) circle(\circ)  node[label=below:$a$] {};
            \draw[fill=black] (b) circle(\circ)  node[label=below:$b$] {};
            \draw[fill=black] (c) circle(\circ)  node[label=below:$c$] {};
            \draw[fill=black] (d) circle(\circ)  node[label=below:$d$] {};
            \draw[fill=black] (e) circle(\circ)  node[label=below:$e$] {};
            \node at (0,-21){(v)};
    \end{tikzpicture}
    \qquad
    \begin{tikzpicture}[scale=0.1]
            \coordinate (x) at (-12,0);
            \coordinate (y) at (12,0);
            \coordinate (a) at (0,18);
            \coordinate (b) at (0,9);
            \coordinate (c) at (0,-9);
            \coordinate (d) at (-21,0);
            \coordinate (e) at (21,0);
            \draw[thick] (x) -- (a) -- (y) -- (x);
            \draw[thick] (x) -- (b) -- (y);
            \draw[thick] (x) -- (c) -- (y);
            \draw[thick] (x) -- (d) -- (y);
            \draw[thick] (x) -- (d);
            \draw[thick] (e) -- (y);
            \draw[thick] (d) -- (a);
            \draw[thick] (d) -- (c);
            \draw[thick] (e) -- (c);
            
            \draw[fill=black] (x) circle(\circ)  node[label=above:$x$] {};
            \draw[fill=black] (y) circle(\circ)  node[label=above:$y$] {};
            \draw[fill=black] (a) circle(\circ)  node[label=below:$a$] {};
            \draw[fill=black] (b) circle(\circ)  node[label=below:$b$] {};
            \draw[fill=black] (c) circle(\circ)  node[label=below:$c$] {};
            \draw[fill=black] (d) circle(\circ)  node[label=below:$d$] {};
            \draw[fill=black] (e) circle(\circ)  node[label=below:$e$] {};
            \node at (0,-21){(vi)};
    \end{tikzpicture}
    \qquad
    \begin{tikzpicture}[scale=0.1]
            \coordinate (x) at (-12,0);
            \coordinate (y) at (12,0);
            \coordinate (a) at (0,18);
            \coordinate (b) at (0,9);
            \coordinate (c) at (0,-9);
            \coordinate (d) at (-21,0);
            \coordinate (e) at (21,0);
            \draw[thick] (x) -- (a) -- (y) -- (x);
            \draw[thick] (x) -- (b) -- (y);
            \draw[thick] (x) -- (c) -- (y);
            \draw[thick] (x) -- (d) -- (y);
            \draw[thick] (x) -- (d);
            \draw[thick] (e) -- (y);
            \draw[thick] (d) -- (a);
            \draw[thick] (d) -- (c);
            \draw[thick] (e)..controls (27,0) and (27,9) ..(b);
            
            \draw[fill=black] (x) circle(\circ)  node[label=above:$x$] {};
            \draw[fill=black] (y) circle(\circ)  node[label=above:$y$] {};
            \draw[fill=black] (a) circle(\circ)  node[label=below:$a$] {};
            \draw[fill=black] (b) circle(\circ)  node[label=below:$b$] {};
            \draw[fill=black] (c) circle(\circ)  node[label=below:$c$] {};
            \draw[fill=black] (d) circle(\circ)  node[label=below:$d$] {};
            \draw[fill=black] (e) circle(\circ)  node[label=below:$e$] {};
            \node at (0,-21){(vii)};
    \end{tikzpicture}
    \qquad
    \begin{tikzpicture}[scale=0.1]
            \coordinate (x) at (-12,0);
            \coordinate (y) at (12,0);
            \coordinate (a) at (0,18);
            \coordinate (b) at (0,9);
            \coordinate (c) at (0,-9);
            \coordinate (d) at (-21,0);
            \coordinate (e) at (21,0);
            \draw[thick] (x) -- (a) -- (y) -- (x);
            \draw[thick] (x) -- (b) -- (y);
            \draw[thick] (x) -- (c) -- (y);
            \draw[thick] (x) -- (d) -- (y);
            \draw[thick] (x) -- (d);
            \draw[thick] (e) -- (y);
            \draw[thick] (d) -- (a);
            \draw[thick] (d) -- (c);
            \draw[thick] (e) -- (c);
            \draw[thick] (e) -- (a);
            
            \draw[fill=black] (x) circle(\circ)  node[label=above:$x$] {};
            \draw[fill=black] (y) circle(\circ)  node[label=above:$y$] {};
            \draw[fill=black] (a) circle(\circ)  node[label=below:$a$] {};
            \draw[fill=black] (b) circle(\circ)  node[label=below:$b$] {};
            \draw[fill=black] (c) circle(\circ)  node[label=below:$c$] {};
            \draw[fill=black] (d) circle(\circ)  node[label=below:$d$] {};
            \draw[fill=black] (e) circle(\circ)  node[label=below:$e$] {};
            \node at (0,-21){(viii)};
    \end{tikzpicture}
    \qquad
    \begin{tikzpicture}[scale=0.1]
            \coordinate (x) at (-12,0);
            \coordinate (y) at (12,0);
            \coordinate (a) at (0,18);
            \coordinate (b) at (0,9);
            \coordinate (c) at (0,-9);
            \coordinate (d) at (-21,0);
            \coordinate (e) at (21,0);
            \draw[thick] (x) -- (a) -- (y) -- (x);
            \draw[thick] (x) -- (b) -- (y);
            \draw[thick] (x) -- (c) -- (y);
            \draw[thick] (x) -- (d) -- (y);
            \draw[thick] (x) -- (d);
            \draw[thick] (e) -- (y);
            \draw[thick] (d) -- (a);
            \draw[thick] (d) -- (c);
            \draw[thick] (e) -- (a);
            \draw[thick] (e)..controls (27,0) and (27,9) ..(b);
            \draw[fill=black] (x) circle(\circ)  node[label=above:$x$] {};
            \draw[fill=black] (y) circle(\circ)  node[label=above:$y$] {};
            \draw[fill=black] (a) circle(\circ)  node[label=below:$a$] {};
            \draw[fill=black] (b) circle(\circ)  node[label=below:$b$] {};
            \draw[fill=black] (c) circle(\circ)  node[label=below:$c$] {};
            \draw[fill=black] (d) circle(\circ)  node[label=below:$d$] {};
            \draw[fill=black] (e) circle(\circ)  node[label=below:$e$] {};
            \node at (0,-21){(ix)};
    \end{tikzpicture}
\captionsetup{singlelinecheck=off}
\caption[foo bar2]{The graph $G$ has a $5-5$ edge $xy$, with $3$ triangles sitting on the edge $xy$.
    \begin{enumerate}[label=(\roman*)]
        \item The vertices $d$ and $e$ have no neighbors in $S_1$.
        \item The vertex $d$ has one neighbor in $S_1$, and $e$ has none.
        \item The vertices $d$ and $e$ have one common neighbor in $S_1$.
        \item The vertices $d$ and $e$ have one distinct neighbor in $S_1$.
        \item The vertex $d$ has two neighbors in $S_1$, while $e$ has none.
        \item The vertex $d$ is the neighbor of $a$ and $c$, and while $e$ is the neighbor of $c$.
        \item The vertex $d$ is the neighbor of $a$ and $c$, while $e$ is the neighbor of $b$.
        \item The vertices $d$ and $e$ are neighbors of $a$ and $c$ both.  
        \item The vertex $d$ is the neighbor of $a$ and $c$, while $e$ is the neighbor of $a$ and $b$.
    \end{enumerate}
}
\label{fig_s_2,4_5,5}
\end{figure}
\begin{claim}
If $G$ contains a $5-5$ edge, then $e(G)\leq \frac{8}{3}n.$
\end{claim}
\begin{proof}
Let $xy\in E(G)$ be a $5-5$ edge.  There are at least $3$ triangles sitting on the edge $xy$, otherwise $G$ contains an $S_{2,4}$.  We subdivide the cases based on the number of triangles sitting on the edge $xy$.
\begin{enumerate}
    \item \textbf{There are $4$ triangles sitting on the edge $xy$.} Let $a,b,c$ and $d$ be the vertices in $G$ which are adjacent to both $x$ and $y$ (see Figure \ref{fig_s_2,4_5,5,4triangles}).  Let $S_1=\{a,b,c,d\}$ and $H=\{x,y,a,b,c,d\}$.  Delete the vertices in $H$. The vertices in $S_1$ can have at most one neighbor in $V (G)\backslash H$ each and can form a path of length $3$ in $S_1$. Hence, the number of edges be deleted is at most $9 + 4 + 3 = 16$. Using the induction hypothesis, $e(G) \leq e(G-H) + 16 \leq \frac{8}{3}(n-6)+16=\frac{8n}{3}$.
    
    \item \textbf{There are $3$ triangles sitting on the edge $xy$.}  Let $a,b$ and $c$ be the vertices in $G$ which are adjacent to both $x$ and $y$. Let $d$ be the vertex adjacent to $x$ but not adjacent to $y$, and $e$ be the vertex adjacent to $y$ but not adjacent to $x$.   Let $S_1=\{a,b,c\}$ and $H=\{x,y, d, e\}\cup S_1$.  Delete the vertices in $H$.  The vertices $d$ and $e$ can have at most one neighbor in $V(G)\backslash H$ each.  We distinguish the cases as follows:
     
     \begin{enumerate}
         \item \textbf{The vertices $d$ and $e$ have no neighbors in $S_1$}, see Figure \ref{fig_s_2,4_5,5}(i).  The vertices in $S_1$ can have at most one neighbor in $V(G)\backslash H$ each and can form a path of length $2$ in $S_1$.  If the vertices $d$ and $e$ are adjacent, they cannot have a neighbor in $V(G)\backslash H$.  Otherwise, we have an $S_{2,4}$ with $dx$ (or $ey$) as the backbone.  Thus, the number of edges deleted is at most $9+2+3+2=16$. 
         
         \item \textbf{One of the vertices $d$ or $e$ has one neighbor in $S_1$, and the other has none.}
         Without loss of generality, suppose $a$ is the neighbor of $d$, see Figure \ref{fig_s_2,4_5,5}(ii).  Note that $a$ cannot have a neighbor in $V(G)\backslash H$.  Otherwise, we have an $S_{2,4}$ with $ay$ as the backbone.  If the vertices $d$ and $e$ are adjacent, they cannot have a neighbor in $V(G)\backslash H$.  Similarly, as before, the vertices $b$ and $c$ can have at most one neighbor in $V(G)\backslash H$ each and the vertices $\{a,b,c\}$ can form a path of length $2$ in $S_1$.  Thus, the number of edges deleted is at most $10+2+2+2=16$.  
         
         \item \textbf{The vertices $d$ and $e$ have one neighbor in $S_1$.}
         There are two possibilities. In the first case, without loss of generality, suppose $a$ is the common neighbor of $d$ and $e$, see Figure \ref{fig_s_2,4_5,5}(iii).  Similarly, as before, $a$ cannot have a neighbor in $V(G)\backslash H$.  If the vertices $d$ and $e$ are adjacent, they cannot have a neighbor in $V(G)\backslash H$.  The vertices $b$ and $c$ can have at most one neighbor in $V(G)\backslash H$ each, and the vertices $\{a,b,c\}$ can form a path of length $2$ in $S_1$.  Thus, the number of edges deleted is at most $11+2+2+2=17$.
         
         Without loss of generality, suppose $a$ is the neighbor of $d$ while $c$ is the neighbor of $e$, see Figure \ref{fig_s_2,4_5,5}(iv).  Similarly, as before, the vertices $a$ and $c$ cannot have a neighbor in $V(G)\backslash H$.  The vertex $b$ can have at most one neighbor in $V(G)\backslash H$ and the vertices $\{a,b,c\}$ can form a path of length $2$ in $S_1$.  Thus, the number of edges deleted is at most $11+2+1+2=16$. (If the vertices $d$ and $e$ are adjacent, there can only be a path of length $1$ inside $S_1$. Again, this precision is unnecessary. We skip this in the following cases also.)
         
         \item \textbf{One of the vertices $d$ or $e$ has two neighbors in $S_1$, while the other has none.}
         Without loss of generality, suppose $d$ is the neighbor of $a$ and $c$, see Figure \ref{fig_s_2,4_5,5}(v).  Similarly, as before, the vertices $a$ and $c$ cannot have a neighbor in $V(G)\backslash H$.  The vertex $b$ can have at most one neighbor in $V(G)\backslash H$ and the vertices $\{a,b,c\}$ can form a path of length $2$ in $S_1$.  If the vertices $d$ and $e$ are adjacent, they cannot have a neighbor in $V(G)\backslash H$.  Thus, the number of edges deleted is at most $11+2+1+2=16$.
         
         \item \textbf{One of the vertices $d$ or $e$ has two neighbors in $S_1$, while the other has one neighbor.}
         There are two possibilities. In the first case, without loss of generality, suppose $d$ is the neighbor of $a$ and $c$, and $e$ is the neighbor of $c$ (see Figure \ref{fig_s_2,4_5,5}(vi)).  Similarly, as before, the vertices $a$ and $c$ cannot have a neighbor in $V(G)\backslash H$.  The vertex $b$ can have at most one neighbor in $V(G)\backslash H$ and the vertices $\{a,b,c\}$ can form a path of length $2$ in $S_1$.  If the vertices $d$ and $e$ are adjacent, they cannot have a neighbor in $V(G)\backslash H$. Thus, the number of edges deleted is at most $12+2+1+2=17$.
         
         In the other case, without loss of generality, assume that $d$ is the neighbor of $a$ and $c$, and $e$ is the neighbor of $b$ (see Figure \ref{fig_s_2,4_5,5}(vii)).  The vertices $a,b$ and $c$ cannot have a neighbor in $V(G)\backslash H$, but they can form a path of length $2$ in $S_1$. Thus, the number of edges deleted is at most $12+2+2=16$.
         
         \item \textbf{Both the vertices $d$ and $e$ have two neighbors in $S_1$.}  There are two possibilities. In the first case, without loss of generality, suppose $d$ and $e$ are the neighbors of $a$ and $c$ both (see Figure \ref{fig_s_2,4_5,5}(viii)).  Similarly, as before, the vertices $a$ and $c$ cannot have a neighbor in $V(G)\backslash H$.  The vertex $b$ can have one neighbor in $V(G)\backslash H$.  The vertices $\{a,b,c\}$ can form a path of length $2$ in $S_1$.  If the vertices $d$ and $e$ are adjacent, they cannot have a neighbor in $V(G)\backslash H$.  Thus, the number of edges deleted is at most $13+1+2+2=18$.
         
         On the other hand, without loss of generality, assume $d$ is the neighbor of $a$ and $c$, while $e$ is the neighbor of $a$ and $b$ (see Figure \ref{fig_s_2,4_5,5}(ix)).  The vertices $a,b$ and $c$ cannot have a neighbor in $V(G)\backslash H$, but they can form a path of length $2$ in $S_1$.  Thus, the  number of edges deleted is at most $13+2+2=17$.  
         
     \end{enumerate}
\end{enumerate}

Using the induction hypothesis, $$e(G)\leq e(G-H)+18\leq \frac{8}{3}(n-7)+18=\frac{8}{3}n.$$  This completes the proof.
\end{proof}
Take $x,y\in V(G)$.  By the previous claims, if $d(x)+d(y)\geq 10$, we are done by induction.  Assume that $d(x)+d(y)\leq 9$.  Summing it over all the edge pairs in $G$, we have $9e\geq \sum_{xy\in E(G)}\left(d(x)+d(y)\right)=\sum_{x\in V(G)}(d(x))^2\geq n\overline{d}^2=n(\frac{2e}{n})^2$, where $\overline{d}$ is the average degree in $G$.  This gives us $e\leq \frac{9}{4}n\leq\frac{8}{3}n$ for $n\geq 1$.  
\end{proof}

\section{Planar Tur\'an number of \texorpdfstring{$S_{2,5}$}{S2,5}}
Let $G$ be an $S_{2,5}$-free plane graph on $n$ vertices.  Let $12|n$. Consider the plane graph consisting of $\frac{n}{12}$ disjoint copies of $5$-regular maximal planar graphs with $12$ vertices, see Figure \ref{c3}.  This graph does not contain an $S_{2,5}$, since it is a $5-$regular graph. Hence, $\text{ex}_{\mathcal{P}}(n,S_{2,4})\geq\frac{5}{2}n$.

\begin{claim}
Let $G$ be an $S_{2,5}$ on $n$ $(1\leq n\leq 42)$ vertices.  The number of edges in $G$ is at most $\frac{20}{7}n$.
\end{claim}
\begin{proof}
Recall that, an $n$-vertex maximal planar graph contains $3n-6$ edges. Since $3n-6\leq \frac{20}{7}n$ for $n\leq 42$,  we get $e(G)\leq \frac{20}{7}n$ holds for all $n$, when $1\leq n\leq 42$.
\end{proof}

\begin{figure}[ht]
\centering
\begin{tikzpicture}[scale=0.04]
\draw[ thick] (0,0) circle (40cm);
\draw[fill=black](0, 40)circle(50pt);
\draw[fill=black](-35,-20)circle(50pt);
\draw[fill=black](35,-20)circle(50pt);
\draw[fill=black](0, 26)circle(50pt);
\draw[fill=black](0,-26)circle(50pt);
\draw[fill=black](23,13)circle(50pt);
\draw[fill=black](-23,13)circle(50pt);
\draw[fill=black](-23,-13)circle(50pt);
\draw[fill=black](23,-13)circle(50pt);
\draw[fill=black](0, -13)circle(50pt);
\draw[fill=black](11,6.5)circle(50pt);
\draw[fill=black](-11,6.5)circle(50pt);
\draw[ thick](0,40)--(23,13)--(35,-20)--(0,-26)--(-35,-20)--(-23,13)--(0,40)(0,26)--(23,13)--(23,-13)--(0,-26)--(-23,-13)--(-23,13)--(0,26)(0,40)--(0,26)(35,-20)--(23,-13)(-35,-20)--(-23,-13)(0,-13)--(11,6.5)--(-11,6.5)--(0,-13)(11,6.5)--(23,13)(-11,6.5)--(-23,13)(0,-13)--(0,-26)(0,26)--(11,6.5)(0,26)--(-11,6.5)(0,-13)--(23,-13)--(11,6.5)(0,-13)--(-23,-13)--(-11,6.5);
\end{tikzpicture}\qquad
\begin{tikzpicture}[scale=0.04]
\draw[ thick] (0,0) circle (40cm);
\draw[fill=black](0, 40)circle(50pt);
\draw[fill=black](-35,-20)circle(50pt);
\draw[fill=black](35,-20)circle(50pt);
\draw[fill=black](0, 26)circle(50pt);
\draw[fill=black](0,-26)circle(50pt);
\draw[fill=black](23,13)circle(50pt);
\draw[fill=black](-23,13)circle(50pt);
\draw[fill=black](-23,-13)circle(50pt);
\draw[fill=black](23,-13)circle(50pt);
\draw[fill=black](0, -13)circle(50pt);
\draw[fill=black](11,6.5)circle(50pt);
\draw[fill=black](-11,6.5)circle(50pt);
\draw[ thick](0,40)--(23,13)--(35,-20)--(0,-26)--(-35,-20)--(-23,13)--(0,40)(0,26)--(23,13)--(23,-13)--(0,-26)--(-23,-13)--(-23,13)--(0,26)(0,40)--(0,26)(35,-20)--(23,-13)(-35,-20)--(-23,-13)(0,-13)--(11,6.5)--(-11,6.5)--(0,-13)(11,6.5)--(23,13)(-11,6.5)--(-23,13)(0,-13)--(0,-26)(0,26)--(11,6.5)(0,26)--(-11,6.5)(0,-13)--(23,-13)--(11,6.5)(0,-13)--(-23,-13)--(-11,6.5);
\end{tikzpicture}\qquad
\begin{tikzpicture}[scale=0.04]
\draw[ thick] (0,0) circle (40cm);
\draw[fill=black](0, 40)circle(50pt);
\draw[fill=black](-35,-20)circle(50pt);
\draw[fill=black](35,-20)circle(50pt);
\draw[fill=black](0, 26)circle(50pt);
\draw[fill=black](0,-26)circle(50pt);
\draw[fill=black](23,13)circle(50pt);
\draw[fill=black](-23,13)circle(50pt);
\draw[fill=black](-23,-13)circle(50pt);
\draw[fill=black](23,-13)circle(50pt);
\draw[fill=black](0, -13)circle(50pt);
\draw[fill=black](11,6.5)circle(50pt);
\draw[fill=black](-11,6.5)circle(50pt);
\draw[fill=black](55,0)circle(20pt);
\draw[fill=black](60,0)circle(20pt);
\draw[fill=black](65,0)circle(20pt);
\draw[ thick](0,40)--(23,13)--(35,-20)--(0,-26)--(-35,-20)--(-23,13)--(0,40)(0,26)--(23,13)--(23,-13)--(0,-26)--(-23,-13)--(-23,13)--(0,26)(0,40)--(0,26)(35,-20)--(23,-13)(-35,-20)--(-23,-13)(0,-13)--(11,6.5)--(-11,6.5)--(0,-13)(11,6.5)--(23,13)(-11,6.5)--(-23,13)(0,-13)--(0,-26)(0,26)--(11,6.5)(0,26)--(-11,6.5)(0,-13)--(23,-13)--(11,6.5)(0,-13)--(-23,-13)--(-11,6.5);
\end{tikzpicture}\qquad
\begin{tikzpicture}[scale=0.04]
\draw[ thick] (0,0) circle (40cm);
\draw[fill=black](0, 40)circle(50pt);
\draw[fill=black](-35,-20)circle(50pt);
\draw[fill=black](35,-20)circle(50pt);
\draw[fill=black](0, 26)circle(50pt);
\draw[fill=black](0,-26)circle(50pt);
\draw[fill=black](23,13)circle(50pt);
\draw[fill=black](-23,13)circle(50pt);
\draw[fill=black](-23,-13)circle(50pt);
\draw[fill=black](23,-13)circle(50pt);
\draw[fill=black](0, -13)circle(50pt);
\draw[fill=black](11,6.5)circle(50pt);
\draw[fill=black](-11,6.5)circle(50pt);
\draw[ thick](0,40)--(23,13)--(35,-20)--(0,-26)--(-35,-20)--(-23,13)--(0,40)(0,26)--(23,13)--(23,-13)--(0,-26)--(-23,-13)--(-23,13)--(0,26)(0,40)--(0,26)(35,-20)--(23,-13)(-35,-20)--(-23,-13)(0,-13)--(11,6.5)--(-11,6.5)--(0,-13)(11,6.5)--(23,13)(-11,6.5)--(-23,13)(0,-13)--(0,-26)(0,26)--(11,6.5)(0,26)--(-11,6.5)(0,-13)--(23,-13)--(11,6.5)(0,-13)--(-23,-13)--(-11,6.5);
\end{tikzpicture}
\caption{$\left(n/12\right)$-disjoint copies of $5$-regular maximal planar graphs with $12$ vertices.}
\label{c3}
\end{figure}
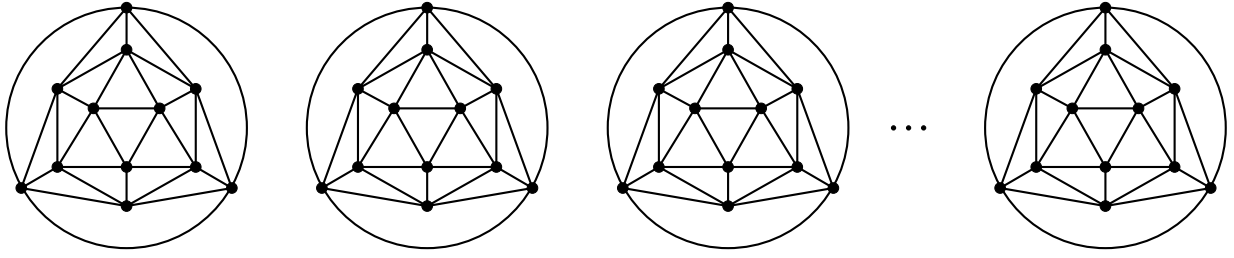

\begin{lemma}\label{deg8_S_2,5}
If $G$ contains a vertex $v$ of degree greater than or equal to $8$, then $e(G)\leq \frac{20}{7}n.$
\end{lemma}
\begin{proof}
Let $v\in V(G)$, such that $d(v)\geq 8$.  It is easy to check that for each $u\in N(v)$, $d(u)\leq 2$. Otherwise, we find a copy of $S_{2,5}$ in $G$. Delete a vertex $u\in N(v)$, then the number of deleted edges is at most $2$. By the induction hypothesis, we get $e(G-\{u\})\leq \frac{20}{7}(n-1)$. Hence, $e(G)=e(G-\{u\})+d(u)\leq \frac{20}{7}(n-1)+2\leq\frac{20}{7}n$.
\end{proof}

\begin{lemma}\label{deg7_S_2,5}
If $G$ contains a vertex $v$ of degree equal to $7$, then $e(G)\leq \frac{20}{7}n.$
\end{lemma}
\begin{proof}
Let $v\in V(G)$, such that $d(v)=7$ and let $H=N(v)\cup \{v\}$.  It is easy to check that for each $u\in N(v)$, if $u$ is adjacent to any vertex in $V(G)\backslash H$, then $d(u)=2$. Otherwise, $G$ contains an $S_{2,5}$.  We are done by induction in this case.  In the other case, assume that for each $u\in N(v)$, $u$ is not adjacent to any vertex in $V(G)\backslash H$.  Then $e(G[H])\leq 3\cdot 8-6=18$. Hence, $e(G)=e(G-H)+e(G[H])\leq \frac{20}{7}(n-8)+18=\frac{20}{7}n$. We are done by induction. 
\end{proof}

\begin{proof}[Proof of Theorem \ref{doublestarsmain}(iv)]
If $G$ contains a vertex of degree at least $7$, we are done by Lemmas \ref{deg8_S_2,5} and \ref{deg7_S_2,5}. Hence, we can assume $\Delta(G)\leq 6$.

\begin{figure}[ht]
\centering
    \begin{tikzpicture}[scale=.1]
            \coordinate (x) at (-12,0);
            \coordinate (y) at (12,0);
            \coordinate (a) at (0,36);
            \coordinate (b) at (0,27);
            \coordinate (c) at (0,18);
            \coordinate (d) at (0,9);
            \coordinate (e) at (0,-9);
            \draw[thick] (x) -- (a) -- (y) -- (x);
            \draw[thick] (x) -- (b) -- (y);
            \draw[thick] (x) -- (c) -- (y);
            \draw[thick] (x) -- (d) -- (y);
            \draw[thick] (x) -- (e) -- (y);
            \draw[fill=black] (x) circle(\circ)  node[label=below:$x$] {};
            \draw[fill=black] (y) circle(\circ)  node[label=below:$y$] {};
            \draw[fill=black] (a) circle(\circ)  node[label=above:$a$] {};
            \draw[fill=black] (b) circle(\circ)  node[label=above:$b$] {};
            \draw[fill=black] (c) circle(\circ)  node[label=above:$c$] {};
            \draw[fill=black] (d) circle(\circ)  node[label=above:$d$] {};
            \draw[fill=black] (e) circle(\circ)  node[label=below:$e$] {};
    \end{tikzpicture}
    \caption{The graph $G$ has a $6-6$ edge $xy$, with $5$ triangles sitting on the edge $xy$.}
    \label{fig_s_3,4_6,6,5triangles}
\end{figure}
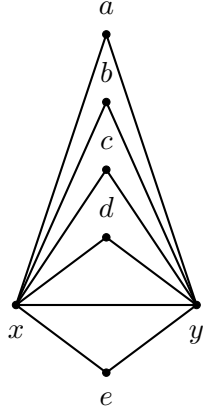
\begin{figure}
\centering
    \begin{tikzpicture}[scale=.1]
            \coordinate (x) at (-12,0);
            \coordinate (y) at (12,0);
            \coordinate (a) at (0,27);
            \coordinate (b) at (0,18);
            \coordinate (c) at (0,9);
            \coordinate (d) at (0,-9);
            \coordinate (f) at (21,0);
            \coordinate (e) at (-21,0);
            \draw[thick] (x) -- (a) -- (y) -- (x);
            \draw[thick] (x) -- (b) -- (y);
            \draw[thick] (x) -- (c) -- (y);
            \draw[thick] (x) -- (d) -- (y);
            \draw[thick] (x) -- (e);
            \draw[thick] (y) -- (f);
            \draw[fill=black] (x) circle(\circ)  node[label=above:$x$] {};
            \draw[fill=black] (y) circle(\circ)  node[label=above:$y$] {};
            \draw[fill=black] (a) circle(\circ)  node[label=below:$a$] {};
            \draw[fill=black] (b) circle(\circ)  node[label=below:$b$] {};
            \draw[fill=black] (c) circle(\circ)  node[label=below:$c$] {};
            \draw[fill=black] (d) circle(\circ)  node[label=above:$d$] {};
            \draw[fill=black] (e) circle(\circ)  node[label=below:$e$] {};
            \draw[fill=black] (f) circle(\circ)  node[label=below:$f$] {};
            \node at (0,-18) {$(i)$};
    \end{tikzpicture}\qquad
    \begin{tikzpicture}[scale=.1]
            \coordinate (x) at (-12,0);
            \coordinate (y) at (12,0);
            \coordinate (a) at (0,27);
            \coordinate (b) at (0,18);
            \coordinate (c) at (0,9);
            \coordinate (d) at (0,-9);
            \coordinate (f) at (21,0);
            \coordinate (e) at (-21,0);
            \draw[thick] (x) -- (a) -- (y) -- (x);
            \draw[thick] (x) -- (b) -- (y);
            \draw[thick] (x) -- (c) -- (y);
            \draw[thick] (x) -- (d) -- (y);
            \draw[thick] (x) -- (e);
            \draw[thick] (y) -- (f);
            \draw[thick] (a) -- (e);
            \draw[fill=black] (x) circle(\circ)  node[label=above:$x$] {};
            \draw[fill=black] (y) circle(\circ)  node[label=above:$y$] {};
            \draw[fill=black] (a) circle(\circ)  node[label=below:$a$] {};
            \draw[fill=black] (b) circle(\circ)  node[label=below:$b$] {};
            \draw[fill=black] (c) circle(\circ)  node[label=below:$c$] {};
            \draw[fill=black] (d) circle(\circ)  node[label=above:$d$] {};
            \draw[fill=black] (e) circle(\circ)  node[label=below:$e$] {};
            \draw[fill=black] (f) circle(\circ)  node[label=below:$f$] {};
            \node at (0,-18) {$(ii)$};
    \end{tikzpicture}\qquad
    \begin{tikzpicture}[scale=.1]
            \coordinate (x) at (-12,0);
            \coordinate (y) at (12,0);
            \coordinate (a) at (0,27);
            \coordinate (b) at (0,18);
            \coordinate (c) at (0,9);
            \coordinate (d) at (0,-9);
            \coordinate (f) at (21,0);
            \coordinate (e) at (-21,0);
            \draw[thick] (x) -- (a) -- (y) -- (x);
            \draw[thick] (x) -- (b) -- (y);
            \draw[thick] (x) -- (c) -- (y);
            \draw[thick] (x) -- (d) -- (y);
            \draw[thick] (x) -- (e);
            \draw[thick] (y) -- (f);
            \draw[thick] (a) -- (e);
            \draw[thick] (a) -- (f);
            \draw[fill=black] (x) circle(\circ)  node[label=above:$x$] {};
            \draw[fill=black] (y) circle(\circ)  node[label=above:$y$] {};
            \draw[fill=black] (a) circle(\circ)  node[label=below:$a$] {};
            \draw[fill=black] (b) circle(\circ)  node[label=below:$b$] {};
            \draw[fill=black] (c) circle(\circ)  node[label=below:$c$] {};
            \draw[fill=black] (d) circle(\circ)  node[label=above:$d$] {};
            \draw[fill=black] (e) circle(\circ)  node[label=below:$e$] {};
            \draw[fill=black] (f) circle(\circ)  node[label=below:$f$] {};
            \node at (0,-18) {$(iii)$};
    \end{tikzpicture}\qquad
    \begin{tikzpicture}[scale=.1]
            \coordinate (x) at (-12,0);
            \coordinate (y) at (12,0);
            \coordinate (a) at (0,27);
            \coordinate (b) at (0,18);
            \coordinate (c) at (0,9);
            \coordinate (d) at (0,-9);
            \coordinate (f) at (21,0);
            \coordinate (e) at (-21,0);
            \draw[thick] (x) -- (a) -- (y) -- (x);
            \draw[thick] (x) -- (b) -- (y);
            \draw[thick] (x) -- (c) -- (y);
            \draw[thick] (x) -- (d) -- (y);
            \draw[thick] (x) -- (e);
            \draw[thick] (y) -- (f);
            \draw[thick] (a) -- (e);
            \draw[thick] (d) -- (f);
            \draw[fill=black] (x) circle(\circ)  node[label=above:$x$] {};
            \draw[fill=black] (y) circle(\circ)  node[label=above:$y$] {};
            \draw[fill=black] (a) circle(\circ)  node[label=below:$a$] {};
            \draw[fill=black] (b) circle(\circ)  node[label=below:$b$] {};
            \draw[fill=black] (c) circle(\circ)  node[label=below:$c$] {};
            \draw[fill=black] (d) circle(\circ)  node[label=above:$d$] {};
            \draw[fill=black] (e) circle(\circ)  node[label=below:$e$] {};
            \draw[fill=black] (f) circle(\circ)  node[label=below:$f$] {};
            \node at (0,-18) {$(iv)$};
    \end{tikzpicture}\qquad
    \begin{tikzpicture}[scale=.1]
            \coordinate (x) at (-12,0);
            \coordinate (y) at (12,0);
            \coordinate (a) at (0,27);
            \coordinate (b) at (0,18);
            \coordinate (c) at (0,9);
            \coordinate (d) at (0,-9);
            \coordinate (f) at (21,0);
            \coordinate (e) at (-21,0);
            \draw[thick] (x) -- (a) -- (y) -- (x);
            \draw[thick] (x) -- (b) -- (y);
            \draw[thick] (x) -- (c) -- (y);
            \draw[thick] (x) -- (d) -- (y);
            \draw[thick] (x) -- (e);
            \draw[thick] (y) -- (f);
            \draw[thick] (a) -- (e);
            \draw[thick] (e) -- (d);
            \draw[fill=black] (x) circle(\circ)  node[label=above:$x$] {};
            \draw[fill=black] (y) circle(\circ)  node[label=above:$y$] {};
            \draw[fill=black] (a) circle(\circ)  node[label=below:$a$] {};
            \draw[fill=black] (b) circle(\circ)  node[label=below:$b$] {};
            \draw[fill=black] (c) circle(\circ)  node[label=below:$c$] {};
            \draw[fill=black] (d) circle(\circ)  node[label=above:$d$] {};
            \draw[fill=black] (e) circle(\circ)  node[label=below:$e$] {};
            \draw[fill=black] (f) circle(\circ)  node[label=below:$f$] {};
            \node at (0,-18) {$(v)$};
    \end{tikzpicture}\qquad
    \begin{tikzpicture}[scale=.1]
            \coordinate (x) at (-12,0);
            \coordinate (y) at (12,0);
            \coordinate (a) at (0,27);
            \coordinate (b) at (0,18);
            \coordinate (c) at (0,9);
            \coordinate (d) at (0,-9);
            \coordinate (f) at (21,0);
            \coordinate (e) at (-21,0);
            \draw[thick] (x) -- (a) -- (y) -- (x);
            \draw[thick] (x) -- (b) -- (y);
            \draw[thick] (x) -- (c) -- (y);
            \draw[thick] (x) -- (d) -- (y);
            \draw[thick] (x) -- (e);
            \draw[thick] (y) -- (f);
            \draw[thick] (a) -- (e);
            \draw[thick] (e) -- (d);
            \draw[thick] (f) -- (d);
            \draw[fill=black] (x) circle(\circ)  node[label=above:$x$] {};
            \draw[fill=black] (y) circle(\circ)  node[label=above:$y$] {};
            \draw[fill=black] (a) circle(\circ)  node[label=below:$a$] {};
            \draw[fill=black] (b) circle(\circ)  node[label=below:$b$] {};
            \draw[fill=black] (c) circle(\circ)  node[label=below:$c$] {};
            \draw[fill=black] (d) circle(\circ)  node[label=above:$d$] {};
            \draw[fill=black] (e) circle(\circ)  node[label=below:$e$] {};
            \draw[fill=black] (f) circle(\circ)  node[label=below:$f$] {};
            \node at (0,-18) {$(vi)$};
    \end{tikzpicture}\qquad
    \begin{tikzpicture}[scale=.1]
            \coordinate (x) at (-12,0);
            \coordinate (y) at (12,0);
            \coordinate (a) at (0,27);
            \coordinate (b) at (0,18);
            \coordinate (c) at (0,9);
            \coordinate (d) at (0,-9);
            \coordinate (f) at (21,0);
            \coordinate (e) at (-21,0);
            \draw[thick] (x) -- (a) -- (y) -- (x);
            \draw[thick] (x) -- (b) -- (y);
            \draw[thick] (x) -- (c) -- (y);
            \draw[thick] (x) -- (d) -- (y);
            \draw[thick] (x) -- (e);
            \draw[thick] (y) -- (f);
            \draw[thick] (a) -- (e);
            \draw[thick] (e) -- (d);
            \draw[thick] (f)..controls (27,0) and (27,9) ..(b);
            \draw[fill=black] (x) circle(\circ)  node[label=above:$x$] {};
            \draw[fill=black] (y) circle(\circ)  node[label=above:$y$] {};
            \draw[fill=black] (a) circle(\circ)  node[label=below:$a$] {};
            \draw[fill=black] (b) circle(\circ)  node[label=below:$b$] {};
            \draw[fill=black] (c) circle(\circ)  node[label=below:$c$] {};
            \draw[fill=black] (d) circle(\circ)  node[label=above:$d$] {};
            \draw[fill=black] (e) circle(\circ)  node[label=below:$e$] {};
            \draw[fill=black] (f) circle(\circ)  node[label=below:$f$] {};
            \node at (0,-18) {$(vii)$};
    \end{tikzpicture}\qquad
    \begin{tikzpicture}[scale=.1]
            \coordinate (x) at (-12,0);
            \coordinate (y) at (12,0);
            \coordinate (a) at (0,27);
            \coordinate (b) at (0,18);
            \coordinate (c) at (0,9);
            \coordinate (d) at (0,-9);
            \coordinate (f) at (21,0);
            \coordinate (e) at (-21,0);
            \draw[thick] (x) -- (a) -- (y) -- (x);
            \draw[thick] (x) -- (b) -- (y);
            \draw[thick] (x) -- (c) -- (y);
            \draw[thick] (x) -- (d) -- (y);
            \draw[thick] (x) -- (e);
            \draw[thick] (y) -- (f);
            \draw[thick] (a) -- (e);
            \draw[thick] (e) -- (d);
            \draw[thick] (f) -- (d);
            \draw[thick] (f) -- (a);
            \draw[fill=black] (x) circle(\circ)  node[label=above:$x$] {};
            \draw[fill=black] (y) circle(\circ)  node[label=above:$y$] {};
            \draw[fill=black] (a) circle(\circ)  node[label=below:$a$] {};
            \draw[fill=black] (b) circle(\circ)  node[label=below:$b$] {};
            \draw[fill=black] (c) circle(\circ)  node[label=below:$c$] {};
            \draw[fill=black] (d) circle(\circ)  node[label=above:$d$] {};
            \draw[fill=black] (e) circle(\circ)  node[label=below:$e$] {};
            \draw[fill=black] (f) circle(\circ)  node[label=below:$f$] {};
            \node at (0,-18) {$(viii)$};
    \end{tikzpicture}\qquad
    \begin{tikzpicture}[scale=.1]
            \coordinate (x) at (-12,0);
            \coordinate (y) at (12,0);
            \coordinate (a) at (0,27);
            \coordinate (b) at (0,18);
            \coordinate (c) at (0,9);
            \coordinate (d) at (0,-9);
            \coordinate (f) at (21,0);
            \coordinate (e) at (-21,0);
            \draw[thick] (x) -- (a) -- (y) -- (x);
            \draw[thick] (x) -- (b) -- (y);
            \draw[thick] (x) -- (c) -- (y);
            \draw[thick] (x) -- (d) -- (y);
            \draw[thick] (x) -- (e);
            \draw[thick] (y) -- (f);
            \draw[thick] (a) -- (e);
            \draw[thick] (e) -- (d);
            \draw[thick] (f)..controls (27,0) and (27,9) ..(b);
            \draw[thick] (f) -- (a);
            \draw[fill=black] (x) circle(\circ)  node[label=above:$x$] {};
            \draw[fill=black] (y) circle(\circ)  node[label=above:$y$] {};
            \draw[fill=black] (a) circle(\circ)  node[label=below:$a$] {};
            \draw[fill=black] (b) circle(\circ)  node[label=below:$b$] {};
            \draw[fill=black] (c) circle(\circ)  node[label=below:$c$] {};
            \draw[fill=black] (d) circle(\circ)  node[label=above:$d$] {};
            \draw[fill=black] (e) circle(\circ)  node[label=below:$e$] {};
            \draw[fill=black] (f) circle(\circ)  node[label=below:$f$] {};
            \node at (0,-18) {$(ix)$};
    \end{tikzpicture}
    \begin{tikzpicture}[scale=.1]
            \coordinate (x) at (-12,0);
            \coordinate (y) at (12,0);
            \coordinate (a) at (0,27);
            \coordinate (b) at (0,18);
            \coordinate (c) at (0,9);
            \coordinate (d) at (0,-9);
            \coordinate (f) at (21,0);
            \coordinate (e) at (-21,0);
            \draw[thick] (x) -- (a) -- (y) -- (x);
            \draw[thick] (x) -- (b) -- (y);
            \draw[thick] (x) -- (c) -- (y);
            \draw[thick] (x) -- (d) -- (y);
            \draw[thick] (x) -- (e);
            \draw[thick] (y) -- (f);
            \draw[thick] (a) -- (e);
            \draw[thick] (e) -- (d);
            \draw[thick] (f)..controls (27,0) and (27,9) ..(b);
            \draw[thick] (f)..controls (27,0) and (27,6) ..(c);
            \draw[fill=black] (x) circle(\circ)  node[label=above:$x$] {};
            \draw[fill=black] (y) circle(\circ)  node[label=above:$y$] {};
            \draw[fill=black] (a) circle(\circ)  node[label=below:$a$] {};
            \draw[fill=black] (b) circle(\circ)  node[label=below:$b$] {};
            \draw[fill=black] (c) circle(\circ)  node[label=below:$c$] {};
            \draw[fill=black] (d) circle(\circ)  node[label=above:$d$] {};
            \draw[fill=black] (e) circle(\circ)  node[label=below:$e$] {};
            \draw[fill=black] (f) circle(\circ)  node[label=below:$f$] {};
            \node at (0,-18) {$(x)$};
    \end{tikzpicture}
    \end{figure}
\begin{figure}
    \centering
\captionsetup{singlelinecheck=off}
\caption[foo bar]{The graph $G$ has a $6-6$ edge $xy$, with $4$ triangles sitting on the edge $xy$.
    \begin{enumerate}[label=(\roman*)]
        \item the vertices $e$ and $f$ have no neighbors in $S_1$.
        \item The vertex $e$ has one neighbor in $S_1$ and $f$ has none. 
        \item The vertices $e$ and $f$ have one common neighbor in $S_1$.
        \item The vertices $e$ and $f$ have one distinct neighbor in $S_1$.
        \item The vertex $e$ has two neighbors in $S_1$ and $f$ has none.
        \item The vertex $e$ is the neighbor of $a$ and $d$, and $f$ is the neighbor of $d$.
        \item The vertex $e$ is the neighbor of $a$ and $d$, and $f$ is the neighbor of $b$.
        \item The vertices $e$ and $f$ are neighbors of both $a$ and $d$.
        \item The vertex $e$ is the neighbor of $a$ and $d$, while $f$ is the neighbor of $a$ and $b$.
        \item The vertex $e$ is the neighbor of $a$ and $d$, while $f$ is the neighbor of $b$ and $c$.
    \end{enumerate}
\ContinuedFloat 
}
\label{fig_s_3,4_6,6}
\end{figure}

\begin{claim} \label{clm6-6s_2,5}
If $G$ contains a $6-6$ edge, then $e(G)\leq \frac{20}{7}n$.
\end{claim}
\begin{proof}
 Let $xy\in E(G)$ be a $6$-$6$ edge. There are at least $4$ triangles sitting on the edge $xy$, otherwise $G$ contains an $S_{2,5}$.  We subdivide the cases based on the number of triangles sitting on the edge $xy$.
 \begin{enumerate}
     \item \textbf{There are $5$ triangles sitting on the edge $xy$.} Let $a,b,c,d$ and $e$ be the vertices in $G$ which are adjacent to both $x$ and $y$ (see Figure \ref{fig_s_3,4_6,6,5triangles}).  Let $S_1=\{a,b,c,d,e\}$ and $H=\{x,y,a,b,c,d,e\}$.  Delete the vertices in $H$.  The vertices in $S_1$ can have at most one neighbor in $V(G)\backslash H$ each and can form a path of length $4$ in $S_1$.  Hence, the number of edges deleted is at most $11+5+4=20$. Using the induction hypothesis, $e(G)\leq e(G-H)+20\leq \frac{20}{7}(n-7)+20=\frac{20}{7}n.$

     \item \textbf{There are $4$ triangles sitting on the edge $xy$.}  Let $a,b,c$ and $d$ be the vertices in $G$ which are adjacent to both $x$ and $y$. Let $e$ be the vertex adjacent to $x$ but not adjacent to $y$, and $f$ be the vertex adjacent to $y$ but not adjacent to $x$.   Let $S_1=\{a,b,c,d\}$ and $H=\{x,y\}\cup S_1\cup\{e,f\}$.  Delete the vertices in $H$.  The vertices $e$ and $f$ can have at most one neighbor in $V(G)\backslash H$ each.  We distinguish the cases based on the neighbors of $e$ and $f$ as follows:
     \begin{enumerate}
         \item \textbf{The vertices $e$ and $f$ have no neighbors in $S_1$}, see Figure \ref{fig_s_3,4_6,6}(i).
         The vertices in $S_1$ can have at most one neighbor in $V(G)\backslash H$ each and can form a path of length $3$ in $S_1$.  If the vertices $e$ and $f$ are adjacent, then they cannot have a neighbor in $V(G)\backslash H$.  Otherwise, we have an $S_{2,5}$ with $ex$ (or $fy$) as the backbone.  Thus, the number of edges deleted is at most $11+3+4+2=20$.  
         
         \item \textbf{One of the vertices $e$ or $f$ has one neighbor in $S_1$, and the other has none.}
         Without loss of generality, suppose $a$ and $e$ are adjacent (see Figure \ref{fig_s_3,4_6,6}(ii)), then $a$ cannot have a neighbor in $V(G)\backslash H$.  Otherwise, we have an $S_{2,5}$ with $ay$ as the backbone.  If the vertices $e$ and $f$ are adjacent, then they cannot have a neighbor in $V(G)\backslash H$.  Similarly, as before, the vertices $b,c$ and $d$ can have at most one neighbor in $V(G)\backslash H$ each and the vertices $\{a,b,c,d\}$ can form a path of length $3$ in $S_1$.  Thus, the number of edges deleted is at most $12+3+3+2=20$.  
         
         \item \textbf{The vertices $e$ and $f$ have one neighbor in $S_1$.}
         There are two possibilities. In the first case, without loss of generality, suppose $a$ is the common neighbor of $e$ and $f$, see Figure \ref{fig_s_3,4_6,6}(iii).  Similarly, as before, $a$ cannot have a neighbor in $V(G)\backslash H$.  If the vertices $e$ and $f$ are adjacent, they cannot have a neighbor in $V(G)\backslash H$.  The vertices $b,c$ and $d$ can have at most one neighbor in $V(G)\backslash H$ each, and the vertices $\{a,b,c,d\}$ can form  a  path of  length  $3$  in $S_1$.  Thus, the number of edges deleted is at most $13+3+3+2=21$. 
         
         Without loss of generality, suppose $a$ is the neighbor of $e$ while $d$ is the neighbor of $f$, see Figure \ref{fig_s_3,4_6,6}(iv). Similarly, as before, the vertices $a$ and $d$ cannot have a neighbor in $V(G)\backslash H$.   The vertex $b$ and $c$ can have at most one neighbor in $V(G)\backslash H$ each,  and  the  vertices $\{a,b,c,d\}$ can form a  path  of  length  $3$  in $S_1$.  If the vertices $e$ and $f$ are adjacent, then they cannot have a neighbor in $V(G)\backslash H$.  Thus, the number of edges deleted is at most $13+3+2+2=20$. (In fact, it can be shown that if the vertices $w$ and $f$ are adjacent, there can only be a $2$-path inside $S_1$.  However, this precision is unnecessary.  We skip this in the following cases also.)
         
         \item \textbf{One of the vertices $e$ or $f$ has two neighbors in $S_1$, while the other has none.}
         Without loss of generality, suppose $e$ is the neighbor of $a$ and $d$, see Figure \ref{fig_s_3,4_6,6}(v). Similarly, as before, the vertices $a$ and $d$ cannot have a neighbor in $V(G)\backslash H$.  The vertices $b$ and $c$ can have at most one neighbor in $V(G)\backslash H$ each, and the vertices $\{a,b,c,d\}$ can form a path of length $3$ in $S_1$.  If the vertices $e$ and $f$ are adjacent, then they cannot have a neighbor in $V(G)\backslash H$.  Thus, the number of edges deleted is at most $13+3+2+2=20$.
         
         \item \textbf{One of the vertices $e$ or $f$ has two neighbors in $S_1$, while the other has one neighbor.}
         There are two possibilities. In the first case, without loss of generality, suppose $e$ is the neighbor of $a$ and $d$, and $f$ is the neighbor of $d$ (see Figure \ref{fig_s_3,4_6,6}(vi)).  Similarly, as before, the vertices $a$ and $d$ cannot have a neighbor in $V(G)\backslash H$.  The vertices $b$ and $c$ can have at most one neighbor in $V(G)\backslash H$ each, and the vertices $\{a,b,c,d\}$ can form a path of length $3$ in $S_1$.  If the vertices $e$ and $f$ are adjacent, then they cannot have a neighbor in $V(G)\backslash H$.  Thus, the number of edges deleted is at most $14+3+2+2=21$.
         
         In the other case, without loss of generality, assume that $e$ is the neighbor of $a$ and $d$, and $f$ is the neighbor of $b$ (see Figure \ref{fig_s_3,4_6,6}(vii)).  The vertices $a,b$ and $d$ cannot have a neighbor in $V(G)\backslash H$, but they along with $c$ can form a path of length $3$ in $S_1$.  The vertex $c$ can have at most one neighbor in $V(G)\backslash H$.  Thus, the number of edges deleted is at most $14+3+1+2=20$.
         
         \item \textbf{Both the vertices $e$ and $f$ have two neighbors in $S_1$.}  
         There are three possibilities. In the first case, without loss of generality, suppose $e$ and $f$ are neighbors of both $a$ and $d$, see Figure \ref{fig_s_3,4_6,6}(viii).  Similarly, as before, the vertices $a$ and $d$ cannot have a neighbor in $V(G)\backslash H$.  The vertices $b$ and $c$ can have at most one neighbor in $V(G)\backslash H$ each.  The vertices $\{a,b,c,d\}$ can form a path of length $3$ in $S_1$.  If the vertices $e$ and $f$ are adjacent, then they cannot have a neighbor in $V(G)\backslash H$.  Thus, the number of edges deleted is at most $15+3+2+2=22$.
         
         On the other hand, without loss of generality, assume $e$ is the neighbor of both $a$ and $d$, while $f$ is the neighbor of $a$ and $b$ (see Figure \ref{fig_s_3,4_6,6}(ix)).  The vertices $a,b$ and $d$ cannot have a neighbor in $V(G)\backslash H$, but they along with $c$ can form a path of length $3$ in $S_1$. The vertex $c$ can have at most one neighbor in $V(G)\backslash H$ each.  Thus, the number of edges deleted is at most $15+3+1+2=21$.  
         
         In the last case, without loss of generality, assume $e$ is the neighbor of $a$ and $d$ both, while $f$ is the neighbor of $b$ and $c$ (see Figure \ref{fig_s_3,4_6,6}(x)).  The vertices $a,b,c$ and $d$ cannot have a neighbor in $V(G)\backslash H$, but they can form a path of length $3$ in $S_1$.  Thus, the  number of edges deleted is at most $15+3+2=20$.
 \end{enumerate}
\end{enumerate}

Using the induction hypothesis, $$e(G)\leq e(G-H)+22\leq \frac{20}{7}(n-8)+22=\frac{20}{7}n-\frac{6}{7}.$$  This completes the proof of the claim.
\end{proof}
  Consider $x,y\in V(G)$.  By the previous claims, if $d(x)+d(y)\geq 12$, we are done by induction.  Assume that $d(x)+d(y)\leq 11$.  Summing it over all the edge pairs in $G$, we have $11e(G)\geq \sum_{xy\in E(G)}\left(d(x)+d(y)\right)=\sum_{x\in V(G)}(d(x))^2\geq n\overline{d}^2=n(\frac{2e(G)}{n})^2$, where $\overline{d}$ is the average degree in $G$.  This gives us $e(G)\leq \frac{11}{4}n\leq \frac{20}{7}n$, for $n\geq 1$.  
\end{proof}

\section{Planar Tur\'an number of \texorpdfstring{$S_{3,3}$}{S3,3}}

We show that for infinitely many integer values of $n$, we can construct an $n$-vertex $S_{3,3}$-free plane graph $G_n$ with $\frac{5}{2}n-5$ edges.  This is to verify that the bound we have is best up to the linear term. Consider a plane graph $G_n$ which is obtained by joining every vertex of the maximal matching on $n-2$ vertices with two vertices. Constructions of $G_n$ when $n$ is even or odd is shown in Figure \ref{exte1}. Each edge in $G_n$ has a degree $3$ end vertex. Thus, $G_n$ is an $S_{3,3}$-free planar graph. Moreover, $e(G_n)=\floor{\frac{5}{2}n}-5$. 


\begin{figure}[ht]
\centering
\captionsetup{justification=centering}
\begin{tikzpicture}[scale=0.5]
\draw[fill=black](-4,0)circle(4pt);
\draw[fill=black](4,0)circle(4pt);
\draw[fill=black](0,-4)circle(4pt);
\draw[fill=black](0,-3)circle(4pt);
\draw[fill=black](0,-2)circle(4pt);
\draw[fill=black](0,2)circle(4pt);
\draw[fill=black](0,0.5)circle(2pt);
\draw[fill=black](0,-0.5)circle(2pt);
\draw[fill=black](0,0)circle(2pt);
\draw[fill=black](0,4)circle(4pt);
\draw[fill=black](0,3)circle(4pt);
\draw[fill=black](0,-1)circle(4pt);
\draw[fill=black](0,1)circle(4pt);
\draw[thick](-4,0)--(0,-4)--(4,0)(-4,0)--(0,-3)--(4,0)(-4,0)--(0,-2)--(4,0)(-4,0)--(0,4)--(4,0)(-4,0)--(0,3)--(4,0)(-4,0)--(0,2)--(4,0)(-4,0)--(0,-1)--(4,0)(-4,0)--(0,1)--(4,0);
\draw[thick](0,4)--(0,3)(0,2)--(0,1) (0,-4)--(0,-3)(0,-2)--(0,-1);
\node at (0,-6){(a) $n$ is even.};
\end{tikzpicture}\qquad\qquad
\begin{tikzpicture}[scale=0.5]
\draw[fill=black](-4,0)circle(4pt);
\draw[fill=black](4,0)circle(4pt);
\draw[fill=black](0,5)circle(4pt);
\draw[fill=black](0,-4)circle(4pt);
\draw[fill=black](0,-3)circle(4pt);
\draw[fill=black](0,-2)circle(4pt);
\draw[fill=black](0,2)circle(4pt);
\draw[fill=black](0,0.5)circle(2pt);
\draw[fill=black](0,-0.5)circle(2pt);
\draw[fill=black](0,0)circle(2pt);
\draw[fill=black](0,4)circle(4pt);
\draw[fill=black](0,3)circle(4pt);
\draw[fill=black](0,-1)circle(4pt);
\draw[fill=black](0,1)circle(4pt);
\draw[thick](-4,0)--(0,5)--(4,0)(-4,0)--(0,-4)--(4,0)(-4,0)--(0,-3)--(4,0)(-4,0)--(0,-2)--(4,0)(-4,0)--(0,4)--(4,0)(-4,0)--(0,3)--(4,0)(-4,0)--(0,2)--(4,0)(-4,0)--(0,-1)--(4,0)(-4,0)--(0,1)--(4,0);
\draw[thick](0,4)--(0,3)(0,2)--(0,1) (0,-4)--(0,-3)(0,-2)--(0,-1);
\node at (0,-6){(b) $n$ is odd.};
\end{tikzpicture}
\caption{Extremal Constructions for the lower bound of planar Tur\'an number of $S_{3,3}$. }
\label{exte1}
\end{figure}
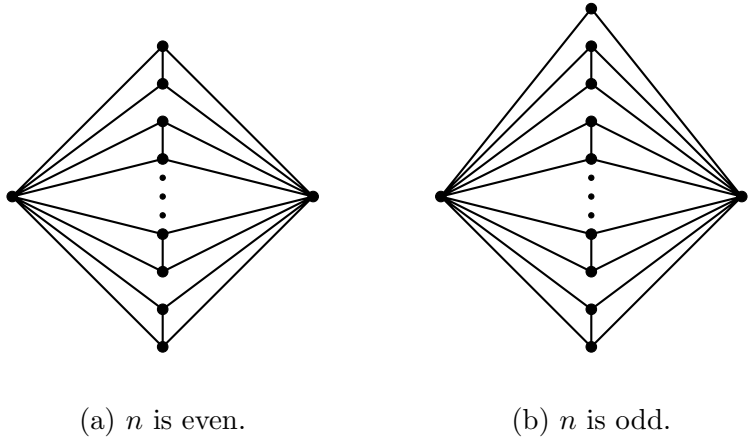

\begin{claim}
Let $G$ be an $S_{3,3}$ on $n (1\leq n\leq 8)$ vertices.  The number of edges in $G$ is at most $\frac{5}{2}n-2$.
\end{claim}
\begin{proof}
Recall that, an $n$-vertex maximal planar graph contains $3n-6$ edges. Since $3n-6\leq \frac{5}{2}n-2$ for $n\leq 8$,  $e(G)\leq \frac{5}{2}n-2$ holds for all $n$, $1\leq n\leq 8$.
\end{proof}

Let $u$ be a vertex in $G$ with degree at most $2$. By the induction hypothesis, we get $e(G-\{u\})\leq \frac{20}{7}(n-1)$. Hence, $e(G)=e(G-\{u\})+d(u)\leq \frac{5}{2}(n-1)-2+2<\frac{5}{2}n-2$. Similarly, if we have a $3$-$3$ edge in $G$, we can finish the proof by induction.  From now on, we may assume that $G$ contains no vertex of degree at most $2$ and no $3$-$3$ edge. The following claims deal with the different cases of degree pairs in $G$:
\begin{claim}\label{clm1}
No vertex in $G$ with  degree at least $7$ is adjacent to a vertex of degree at least $4$.
\end{claim}
\begin{proof}
 Suppose not. Let $xy$ be an edge in $G$ such that $d(x)\geq 7$ and $d(y)\geq 4$. Obviously, there are three vertices in $V(G)\backslash \{x\}$, say $y_1,y_2$ and $y_3$, which are adjacent to $y$. Since $|N(x)\backslash\{y\}|\geq 6$, there are three vertices $x_1,x_2$ and $x_3$, not in $\{y,y_1,y_2,y_3\}$ which are adjacent to $x$. This implies that we got an $S_{3,3}$ in $G$ with backbone $xy$ and leaf-sets $\{x_1,x_2,x_3\}$ and $\{y_1,y_2,y_3\}$, respectively, which is a contradiction. 
 This completes the proof of Claim \ref{clm1}.
\end{proof}

\begin{claim}\label{clm6-6}
If there is a $6$-$6$ edge in $G$, then $e(G)\leq \frac{5}{2}n-2.$
\end{claim}
\begin{proof}
 Let $xy\in E(G)$ be a $6$-$6$ edge. Since $G$ is an $S_{3,3}$-free plane graph, $xy$ must be contained in $5$ triangles, see Figure \ref{6-6 case}.
 \begin{figure}[ht]
\centering
\captionsetup{justification=centering}
    \begin{tikzpicture}[scale=.1]
            \coordinate (x) at (-12,0);
            \coordinate (y) at (12,0);
            \coordinate (a) at (0,36);
            \coordinate (b) at (0,27);
            \coordinate (c) at (0,18);
            \coordinate (d) at (0,9);
            \coordinate (e) at (0,-9);
            \draw[thick] (x) -- (a) -- (y) -- (x);
            \draw[thick] (x) -- (b) -- (y);
            \draw[thick] (x) -- (c) -- (y);
            \draw[thick] (x) -- (d) -- (y);
            \draw[thick] (x) -- (e) -- (y);
            \draw[fill=black] (x) circle(\circ)  node[label=below:$x$] {};
            \draw[fill=black] (y) circle(\circ)  node[label=below:$y$] {};
            \draw[fill=black] (a) circle(\circ)  node[label=above:$a$] {};
            \draw[fill=black] (b) circle(\circ)  node[label=above:$b$] {};
            \draw[fill=black] (c) circle(\circ)  node[label=above:$c$] {};
            \draw[fill=black] (d) circle(\circ)  node[label=above:$d$] {};
            \draw[fill=black] (e) circle(\circ)  node[label=below:$e$] {};
    \end{tikzpicture}
\caption{The graph $G$ has a $6-6$ edge $xy$.}
\label{6-6 case}
\end{figure}
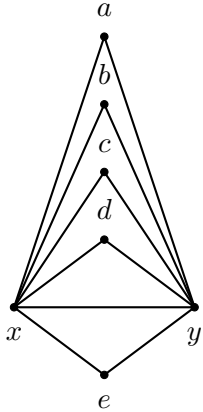
Let $a,b,c,d$ and $e$ be the vertices in $G$ which are adjacent to both $x$ and $y$.  Let $S_1=\{a,b,c,d,e\}$ and $H=\{x,y,a,b,c,d,e\}$.  Delete the vertices in $H$.  Suppose $a$ has two neighbors in $V(G)\backslash H$.  We immediately get an $S_{3,3}$ with $xa$ or $ya$ as the backbone.  Thus, any vertex in the set $S_1$ can have at most $1$ neighbor in $V(G)\backslash H$. If there are no edges between the vertices in $S_1$, we deleted at most $11+5=16$ edges.  Assume that there is an edge between the vertices in $S_1$, say $ab$.  If $a$ (or $b$) has a neighbor in $V(G)\backslash H$, $xa$ (or $xb$) is the backbone of an $S_{3,3}$.  Similarly, for the other edges in $S_1$, both vertices cannot have a neighbor in $V(G)\backslash H$.  Thus, if there is an edge joining any two vertices in $S_1$, the number of edges deleted is at most $11+4=15$.  Using the induction hypothesis, $$e(G)\leq e(G-H)+16\leq \frac{5}{2}(n-7)-2+16=\frac{5}{2}n-3.5<\frac{5}{2}n-2.$$  This completes the proof.
\end{proof}

\begin{figure}
\centering
    \begin{tikzpicture}[scale=.1]
            \coordinate (x) at (-12,0);
            \coordinate (y) at (12,0);
            \coordinate (a) at (0,27);
            \coordinate (b) at (0,18);
            \coordinate (c) at (0,9);
            \coordinate (d) at (0,-9);
            \coordinate (e) at (21,0);
            \draw[thick] (x) -- (a) -- (y) -- (x);
            \draw[thick] (x) -- (b) -- (y);
            \draw[thick] (x) -- (c) -- (y);
            \draw[thick] (x) -- (d) -- (y);
            \draw[thick] (e) -- (y);
            \draw[fill=black] (x) circle(\circ)  node[label=below:$x$] {};
            \draw[fill=black] (y) circle(\circ)  node[label=below:$y$] {};
            \draw[fill=black] (a) circle(\circ)  node[label=below:$a$] {};
            \draw[fill=black] (b) circle(\circ)  node[label=below:$b$] {};
            \draw[fill=black] (c) circle(\circ)  node[label=below:$c$] {};
            \draw[fill=black] (d) circle(\circ)  node[label=above:$d$] {};
            \draw[fill=black] (e) circle(\circ)  node[label=below:$e$] {};
            \node at (0,-21){(i)};
    \end{tikzpicture}\qquad
    \begin{tikzpicture}[scale=.1]
            \coordinate (x) at (-12,0);
            \coordinate (y) at (12,0);
            \coordinate (a) at (0,27);
            \coordinate (b) at (0,18);
            \coordinate (c) at (0,9);
            \coordinate (d) at (0,-9);
            \coordinate (e) at (21,0);
            \draw[thick] (x) -- (a) -- (y) -- (x);
            \draw[thick] (x) -- (b) -- (y);
            \draw[thick] (x) -- (c) -- (y);
            \draw[thick] (x) -- (d) -- (y);
            \draw[thick] (e) -- (y);
            \draw[thick] (e) -- (a);
            \draw[fill=black] (x) circle(\circ)  node[label=below:$x$] {};
            \draw[fill=black] (y) circle(\circ)  node[label=below:$y$] {};
            \draw[fill=black] (a) circle(\circ)  node[label=below:$a$] {};
            \draw[fill=black] (b) circle(\circ)  node[label=below:$b$] {};
            \draw[fill=black] (c) circle(\circ)  node[label=below:$c$] {};
            \draw[fill=black] (d) circle(\circ)  node[label=above:$d$] {};
            \draw[fill=black] (e) circle(\circ)  node[label=below:$e$] {};
            \node at (0,-21){(ii)};
    \end{tikzpicture}\qquad
    \begin{tikzpicture}[scale=.1]
            \coordinate (x) at (-12,0);
            \coordinate (y) at (12,0);
            \coordinate (a) at (0,27);
            \coordinate (b) at (0,18);
            \coordinate (c) at (0,9);
            \coordinate (d) at (0,-9);
            \coordinate (e) at (21,0);
            \draw[thick] (x) -- (a) -- (y) -- (x);
            \draw[thick] (x) -- (b) -- (y);
            \draw[thick] (x) -- (c) -- (y);
            \draw[thick] (x) -- (d) -- (y);
            \draw[thick] (e) -- (y);
            \draw[thick] (e) -- (a);
            \draw[thick] (e) -- (d);
            \draw[fill=black] (x) circle(\circ)  node[label=above:$x$] {};
            \draw[fill=black] (y) circle(\circ)  node[label=above:$y$] {};
            \draw[fill=black] (a) circle(\circ)  node[label=below:$a$] {};
            \draw[fill=black] (b) circle(\circ)  node[label=below:$b$] {};
            \draw[fill=black] (c) circle(\circ)  node[label=below:$c$] {};
            \draw[fill=black] (d) circle(\circ)  node[label=above:$d$] {};
            \draw[fill=black] (e) circle(\circ)  node[label=below:$e$] {};
            \node at (0,-21){(iii)};
    \end{tikzpicture}
\captionsetup{singlelinecheck=off}
\caption[foo bar2]{The graph $G$ has a $5-6$ edge $xy$.
\begin{enumerate}[label=(\roman*)]
    \item The vertex $e$ has no neighbors in $S_1$.
    \item The vertex $e$ has one neighbor in $S_1$.
    \item The vertex $e$ has two neighbors in $S_1$.
\end{enumerate}}
\label{6-5 case}
\end{figure}

\begin{claim}\label{clm5-6}
If there is a $5$-$6$ edge in $G$, then $e(G)\leq \frac{5}{2}n-2.$
\end{claim}
\begin{proof}
Let $xy$ be a $5$-$6$ edge in $G$.  Since $G$ is an $S_{3,3}$-free plane graph, $xy$ must be contained in $4$ triangles, see Figure \ref{6-5 case}.  Let $a,b,c$ and $d$ be the vertices in $G$ which are adjacent to both $x$ and $y$.  Let $e$ be the vertex adjacent to $y$ but not adjacent to $x$. Let $S_1=\{a,b,c,d\}$ and $H=\{x,y,a,b,c,d,e\}$.  Delete the vertices in $H$.  Any vertex in $S_1$ can have at most $1$ neighbor in $V(G)\backslash H$.  If there is an edge joining any two vertices in $S_1$, say $ab$.  Similarly, as before, the vertices $a$ and $b$ cannot have a neighbor in $V(G)\backslash H$. We further distinguish the cases based on the number of edges from $e$ as follows:
\begin{enumerate}
\item \textbf{The vertex $e$ has no neighbors in $S_1$}, see Figure \ref{6-5 case}(i). Clearly, the vertex $e$ can have at most $2$ neighbors in $V(G)\backslash H$. If there are no edges between the vertices in $S_1$, we deleted at most $10+2+4=16$ edges.  If there is an edge joining any two vertices in $S_1$, the number of edges deleted is at most $10+2+3=15$.

\item \textbf{The vertex $e$ has one neighbor in $S_1$.} Without loss of generality, suppose $a$ is the neighbor of $e$, see Figure \ref{6-5 case}(ii).  The vertex $e$ can have at most one neighbor in $V(G)\backslash H$.  If the vertex $a$ has a neighbor in $V(G)\backslash H$, we have an $S_{3,3}$ with $xa$ as the backbone.  If there are no edges between the vertices in $S_1$, we have deleted at most $11+3+1=15$ edges.  If there is an edge joining any two vertices in $S_1$, the number of edges deleted is at most $11+3+1=15$.

\item \textbf{The vertex $e$ has two neighbors in $S_1$.}  Without loss of generality, suppose $e$ is the neighbor of $a$ and $d$, see Figure \ref{6-5 case}(iii). The vertex $e$ cannot have a neighbor in $V(G)\backslash H$, otherwise we have an $S_{3,3}$ with $ye$ as the backbone.  If either $a$ or $d$ has a neighbor in $V(G)\backslash H$, we have an $S_{3,3}$ with $xa$ or $xd$ as the backbone, respectively. Suppose there is no edge between the vertices in $S_1$. The total number of edges deleted is at most $12+2=14$. Suppose $a$ and $b$ are adjacent, then $b$ cannot have a neighbor in $V(G)\backslash H$.  Similarly, for the other edges in $S_1$ except $ad$, which is still possible without any extra restrictions.  If there is an edge joining any two vertices in $S_1$, the total number of edges deleted is at most $12+3=15$. 
\end{enumerate}
Thus, by induction $$e(G)=e(G-H)+16\leq \frac{5}{2}(n-7)-2+16< \frac{5}{2}n-2,$$ and we are done.
\end{proof}

\begin{figure}
\centering
\begin{tikzpicture}[scale=0.1]
            \coordinate (x) at (-12,0);
            \coordinate (y) at (12,0);
            \coordinate (a) at (0,18);
            \coordinate (b) at (0,9);
            \coordinate (c) at (0,-9);
            \coordinate (d) at (21,6); 
            \coordinate (e) at (21,0); 
            \draw[thick] (x) -- (a) -- (y) -- (x);
            \draw[thick] (x) -- (b) -- (y);
            \draw[thick] (x) -- (c) -- (y);
            \draw[thick] (d) -- (y);
            \draw[thick] (e) -- (y);
            \draw[fill=black] (x) circle(\circ)  node[label=below:$x$] {};
            \draw[fill=black] (y) circle(\circ)  node[label=below:$y$] {};
            \draw[fill=black] (a) circle(\circ)  node[label=below:$a$] {};
            \draw[fill=black] (b) circle(\circ)  node[label=below:$b$] {};
            \draw[fill=black] (c) circle(\circ)  node[label=above:$c$] {};
            \draw[fill=black] (d) circle(\circ)  node[label=above:$d$] {};
            \draw[fill=black] (e) circle(\circ)  node[label=below:$e$] {};
            \node at (0,-18){(i)};
    \end{tikzpicture}\qquad
\begin{tikzpicture}[scale=0.1]
            \coordinate (x) at (-12,0);
            \coordinate (y) at (12,0);
            \coordinate (a) at (0,18);
            \coordinate (b) at (0,9);
            \coordinate (c) at (0,-9);
            \coordinate (d) at (21,6); 
            \coordinate (e) at (21,0); 
            \draw[thick] (x) -- (a) -- (y) -- (x);
            \draw[thick] (x) -- (b) -- (y);
            \draw[thick] (x) -- (c) -- (y);
            \draw[thick] (d) -- (y);
            \draw[thick] (d) -- (a);
            \draw[thick] (e) -- (y);
            \draw[fill=black] (x) circle(\circ)  node[label=below:$x$] {};
            \draw[fill=black] (y) circle(\circ)  node[label=below:$y$] {};
            \draw[fill=black] (a) circle(\circ)  node[label=below:$a$] {};
            \draw[fill=black] (b) circle(\circ)  node[label=below:$b$] {};
            \draw[fill=black] (c) circle(\circ)  node[label=above:$c$] {};
            \draw[fill=black] (d) circle(\circ)  node[label=above:$d$] {};
            \draw[fill=black] (e) circle(\circ)  node[label=below:$e$] {};
            \node at (0,-18){(ii)};
    \end{tikzpicture}
\qquad
\begin{tikzpicture}[scale=0.1]
            \coordinate (x) at (-12,0);
            \coordinate (y) at (12,0);
            \coordinate (a) at (0,18);
            \coordinate (b) at (0,9);
            \coordinate (c) at (0,-9);
            \coordinate (d) at (21,6); 
            \coordinate (e) at (21,0); 
            \draw[thick] (x) -- (a) -- (y) -- (x);
            \draw[thick] (x) -- (b) -- (y);
            \draw[thick] (x) -- (c) -- (y);
            \draw[thick] (d) -- (y);
            \draw[thick] (d) -- (a);
            \draw[thick] (e) -- (y);
            \draw[thick](e)..controls (27,12) and (27,18) ..(a);
            \draw[fill=black] (x) circle(\circ)  node[label=below:$x$] {};
            \draw[fill=black] (y) circle(\circ)  node[label=below:$y$] {};
            \draw[fill=black] (a) circle(\circ)  node[label=below:$a$] {};
            \draw[fill=black] (b) circle(\circ)  node[label=below:$b$] {};
            \draw[fill=black] (c) circle(\circ)  node[label=above:$c$] {};
            \draw[fill=black] (d) circle(\circ)  node[label=above:$d$] {};
            \draw[fill=black] (e) circle(\circ)  node[label=below:$e$] {};
            \node at (0,-18){(iii)};
    \end{tikzpicture}
\qquad
\begin{tikzpicture}[scale=0.1]
            \coordinate (x) at (-12,0);
            \coordinate (y) at (12,0);
            \coordinate (a) at (0,18);
            \coordinate (b) at (0,9);
            \coordinate (c) at (0,-9);
            \coordinate (d) at (21,6); 
            \coordinate (e) at (21,0); 
            \draw[thick] (x) -- (a) -- (y) -- (x);
            \draw[thick] (x) -- (b) -- (y);
            \draw[thick] (x) -- (c) -- (y);
            \draw[thick] (d) -- (y);
            \draw[thick] (d) -- (a);
            \draw[thick] (e) -- (y);
            \draw[thick](e)--(c);
            \draw[fill=black] (x) circle(\circ)  node[label=above:$x$] {};
            \draw[fill=black] (y) circle(\circ)  node[label=above:$y$] {};
            \draw[fill=black] (a) circle(\circ)  node[label=below:$a$] {};
            \draw[fill=black] (b) circle(\circ)  node[label=below:$b$] {};
            \draw[fill=black] (c) circle(\circ)  node[label=above:$c$] {};
            \draw[fill=black] (d) circle(\circ)  node[label=above:$d$] {};
            \draw[fill=black] (e) circle(\circ)  node[label=below:$e$] {};
            \node at (0,-18){(iv)};
    \end{tikzpicture}
\qquad
\begin{tikzpicture}[scale=0.1]
            \coordinate (x) at (-12,0);
            \coordinate (y) at (12,0);
            \coordinate (a) at (0,18);
            \coordinate (b) at (0,9);
            \coordinate (c) at (0,-9);
            \coordinate (d) at (21,6); 
            \coordinate (e) at (21,0); 
            \draw[thick] (x) -- (a) -- (y) -- (x);
            \draw[thick] (x) -- (b) -- (y);
            \draw[thick] (x) -- (c) -- (y);
            \draw[thick] (d) -- (y);
            \draw[thick] (d) -- (a);
            \draw[thick] (e) -- (y);
            \draw[thick](d)..controls (27,0) and (27,-12) ..(c);
            \draw[fill=black] (x) circle(\circ)  node[label=above:$x$] {};
            \draw[fill=black] (y) circle(\circ)  node[label=above:$y$] {};
            \draw[fill=black] (a) circle(\circ)  node[label=below:$a$] {};
            \draw[fill=black] (b) circle(\circ)  node[label=below:$b$] {};
            \draw[fill=black] (c) circle(\circ)  node[label=above:$c$] {};
            \draw[fill=black] (d) circle(\circ)  node[label=above:$d$] {};
            \draw[fill=black] (e) circle(\circ)  node[label=below:$e$] {};
            \node at (0,-18){(v)};
    \end{tikzpicture}
\qquad
\begin{tikzpicture}[scale=0.1]
            \coordinate (x) at (-12,0);
            \coordinate (y) at (12,0);
            \coordinate (a) at (0,18);
            \coordinate (b) at (0,9);
            \coordinate (c) at (0,-9);
            \coordinate (d) at (21,6); 
            \coordinate (e) at (21,0); 
            \draw[thick] (x) -- (a) -- (y) -- (x);
            \draw[thick] (x) -- (b) -- (y);
            \draw[thick] (x) -- (c) -- (y);
            \draw[thick] (d) -- (y);
            \draw[thick] (d) -- (a);
            \draw[thick] (e) -- (y);
            \draw[thick] (e) -- (c);
            \draw[thick](d)..controls (27,0) and (27,-12) ..(c);
            \draw[fill=black] (x) circle(\circ)  node[label=above:$x$] {};
            \draw[fill=black] (y) circle(\circ)  node[label=above:$y$] {};
            \draw[fill=black] (a) circle(\circ)  node[label=below:$a$] {};
            \draw[fill=black] (b) circle(\circ)  node[label=below:$b$] {};
            \draw[fill=black] (c) circle(\circ)  node[label=above:$c$] {};
            \draw[fill=black] (d) circle(\circ)  node[label=above:$d$] {};
            \draw[fill=black] (e) circle(\circ)  node[label=below:$e$] {};
            \node at (0,-18){(vi)};
    \end{tikzpicture}
\qquad
\begin{tikzpicture}[scale=0.1]
            \coordinate (x) at (-12,0);
            \coordinate (y) at (12,0);
            \coordinate (a) at (0,18);
            \coordinate (c) at (0,-27);
            \coordinate (b) at (0,-3);
            \coordinate (d) at (21,6); 
            \coordinate (e) at (3,-6); 
            \draw[thick] (x) -- (a) -- (y) -- (x);
            \draw[thick] (x) -- (b) -- (y);
            \draw[thick] (x) -- (c) -- (y);
            \draw[thick] (d) -- (y);
            \draw[thick] (d) -- (a);
            \draw[thick] (e) -- (y);
            \draw[thick] (e) -- (b);
            \draw[thick](d)..controls (27,0) and (27,-12) ..(c);
            \draw[fill=black] (x) circle(\circ)  node[label=above:$x$] {};
            \draw[fill=black] (y) circle(\circ)  node[label=above:$y$] {};
            \draw[fill=black] (a) circle(\circ)  node[label=below:$a$] {};
            \draw[fill=black] (b) circle(\circ)  node[label={[label distance=.01cm]225:$b$}] {};
            \draw[fill=black] (c) circle(\circ)  node[label=below:$c$] {};
            \draw[fill=black] (d) circle(\circ)  node[label=above:$d$] {};
            \draw[fill=black] (e) circle(\circ)  node[label=right:$e$] {};
            \node at (0,-36){(vii)};
    \end{tikzpicture}\qquad
\begin{tikzpicture}[scale=0.1]
            \coordinate (x) at (-12,0);
            \coordinate (y) at (12,0);
            \coordinate (a) at (0,18);
            \coordinate (c) at (0,-27);
            \coordinate (b) at (0,-3);
            \coordinate (d) at (21,6); 
            \coordinate (e) at (3,-6); 
            \draw[thick] (x) -- (a) -- (y) -- (x);
            \draw[thick] (x) -- (b) -- (y);
            \draw[thick] (x) -- (c) -- (y);
            \draw[thick] (d) -- (y);
            \draw[thick] (d) -- (a);
            \draw[thick] (e) -- (y);
            \draw[thick] (e) -- (c);
            \draw[thick] (e) -- (b);
            \draw[thick](d)..controls (27,0) and (27,-12) ..(c);
            \draw[fill=black] (x) circle(\circ)  node[label=above:$x$] {};
            \draw[fill=black] (y) circle(\circ)  node[label=above:$y$] {};
            \draw[fill=black] (a) circle(\circ)  node[label=below:$a$] {};
            \draw[fill=black] (b) circle(\circ)  node[label={[label distance=.01cm]225:$b$}] {};
            \draw[fill=black] (c) circle(\circ)  node[label=below:$c$] {};
            \draw[fill=black] (d) circle(\circ)  node[label=above:$d$] {};
            \draw[fill=black] (e) circle(\circ)  node[label=right:$e$] {};
            \node at (0,-36){(viii)};
    \end{tikzpicture}
\captionsetup{singlelinecheck=off}
\caption[foo bar]{The graph $G$ has a $4-6$ edge $xy$.
\begin{enumerate}[label=(\roman*)]
    \item The vertices $d$ and $e$ have no neighbors in $S_1$.
    \item The vertex $d$ has one neighbor in $S_1$, and $f$ has none.
    \item Both the vertices $d$ and $e$ have one common neighbor in $S_1$.
    \item The vertices $d$ and $e$ have one distinct neighbor in $S_1$
    \item The vertex $d$ has two neighbors in $S_1$, while $e$ has none.
    \item The vertex $d$ is the neighbor of $a$ and $c$, while $e$ is the neighbor of $c$.
    \item The vertex $d$ is the neighbor of $a$ and $c$, while $e$ is the neighbor of $b$.
    \item Both the vertices $d$ and $e$ have two neighbors in $S_1$.
\end{enumerate}}
\label{6-4}
\end{figure}
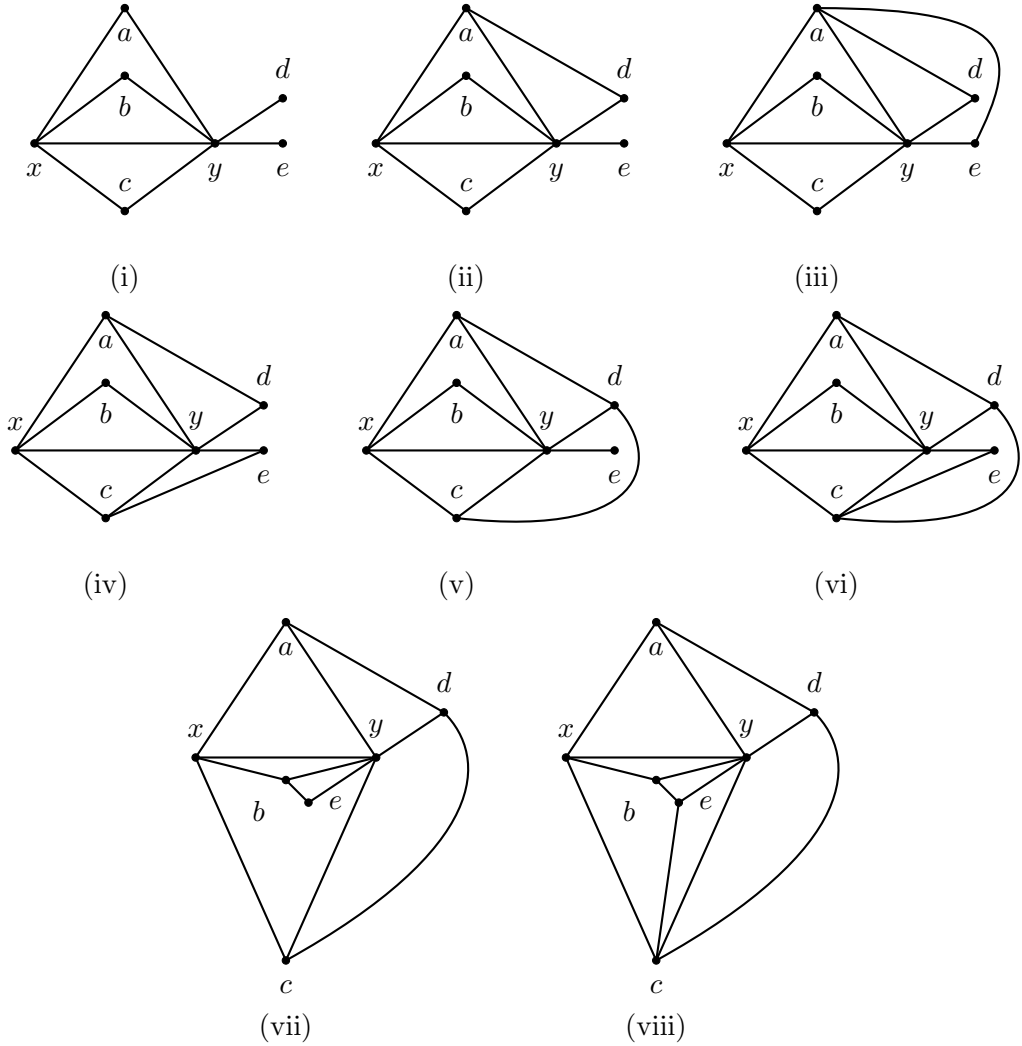

\begin{claim}\label{clm4-6}
If there is a $4$-$6$ edge in $G$, then $e(G)\leq \frac{5}{2}n-2.$
\end{claim}
\begin{proof}
Let $xy$ be a $4-6$ edge in $G$.  Since $G$ is an $S_{3,3}$-free plane graph, $xy$ must be contained in $3$ triangles, see Figure \ref{6-4}.  Let $a,b$ and $c$ be the vertices in $G$ which are adjacent to both $x$ and $y$. Let $d$ and $e$ be the vertices adjacent to $y$ but not to $x$. Let $S_1=\{a,b,c\}$ and $H=\{x,y,a,b,c,d,e\}$.  Delete the vertices in $H$.  The vertices $d$ and $e$ can have at most two neighbors in $V(G)\backslash H$ each.  The vertices in $S_1$ can have at most one neighbor in $V(G)\backslash H$ each.  If there is an edge joining any two vertices in $S_1$, say $ab$.  Similarly, as before, the vertices $a$ and $b$ cannot have a neighbor in $V(G)\backslash H$.  We distinguish the cases based on the neighbors of $d$ and $e$ as follows:
\begin{enumerate}
         \item \textbf{The vertices $d$ and $e$ have no neighbors in $S_1$}, see Figure \ref{6-4}(i).
         If the vertices $d$ and $e$ are adjacent, they can have at most one neighbor in $V(G)\backslash H$ each.  Otherwise, we have an $S_{3,3}$ with $dy$ (or $ey$) as the backbone.  If there are no edges between the vertices in $S_1$, the number of edges deleted is at most $9+3+4=16$.  If there is an edge joining any two vertices in $S_1$, the number of edges deleted is at most $9+2+4=15$.
         
         \item \textbf{One of the vertices $d$ or $e$ has one neighbor in $S_1$, while the other has none.}
         Without loss of generality, suppose the vertices $a$ and $d$ are adjacent (see Figure \ref{6-4}(ii)).  The vertex $a$ cannot have a neighbor in $V(G)\backslash H$, and $d$ can have at most one neighbor in $V(G)\backslash H$.  If the vertices $d$ and $e$ are adjacent, then $d$ cannot have a neighbor in $V(G)\backslash H$ and $e$ can have at most one neighbor in $V(G)\backslash H$.  Otherwise, we have an $S_{3,3}$ with $dy$ (or $ey$) as the backbone.  If there are no edges between the vertices in $S_1$, the  number of edges deleted is at most $10+2+3=15$.  If there is an edge joining any two vertices in $S_1$, the number of edges deleted is at most $10+2+3=15$.  
         
         \item \textbf{Both the vertices $d$ and $e$ have one neighbor in $S_1$.}
         There are two possibilities. In the first case, without loss of generality, suppose $a$ is the common neighbor of $d$ and $e$ (see Figure \ref{6-4}(iii)).  Similarly, as before, $a$ cannot have a neighbor in $V(G)\backslash H$, and $d$ and $e$ can have at most one neighbor in $V(G)\backslash H$ each.  If the vertices $d$ and $e$ are adjacent, they cannot have a neighbor in $V(G)\backslash H$.  If there are no edges between the vertices in $S_1$, the number of edges deleted is at most $11+2+2=15$.  If there is an edge joining any two vertices in $S_1$, the number of edges deleted is at most $11+2+2=15$.
         
         Without loss of generality, suppose $a$ is the neighbor of $d$ while $c$ is the neighbor of $e$ (see Figure \ref{6-4}(iv)). Similarly, as before, the vertices $a$ and $c$ cannot have a neighbor in $V(G)\backslash H$, and $d$ and $e$ can have at most one neighbor in $V(G)\backslash H$ each.  If the vertices $d$ and $e$ are adjacent, they cannot have a neighbor in $V(G)\backslash H$.  If there are no edges between the vertices in $S_1$, the number of edges deleted is at most $11+1+2=14$.   If $a$ and $b$ are adjacent, then $b$ does not have a neighbor in $V(G)\backslash H$.  Similarly, for the edge $bc$.  The vertices $a$ and $c$ can be adjacent without any constraints.  Thus, if there is an edge joining any two vertices in $S_1$, the number of edges deleted is at most $11+2+2=15$.

         \item \textbf{One of the vertices $d$ or $e$ has two neighbors in $S_1$, while the other has none.}  Without loss of generality, suppose $d$ is the neighbor of $a$ and $c$, while $e$ has no neighbors in $S_1$ (see Figure \ref{6-4}(v)).  The vertex $d$ cannot have a neighbor in $V(G)\backslash H$.  Similarly, as before, the vertices $a$ and $c$ cannot have a neighbor in $V(G)\backslash H$, whereas $e$ can have at most two neighbors in $V(G)\backslash H$.  If $d$ and $e$ are adjacent, then $e$ can have at most one neighbor in $V(G)\backslash H$.  If there are no edges between the vertices in $S_1$, the number of edges deleted is at most $11+1+2=14$.  If $a$ and $b$ are adjacent, then $b$ does not have a neighbor in $V(G)\backslash H$.  Similarly, for the edge $bc$.  The vertices $a$ and $c$ can be adjacent without any constraints.  Thus, if there is an edge joining any two vertices in $S_1$, the number of edges deleted is at most $11+2+2=15$.

         \item \textbf{One of the vertices $d$ or $e$ has two neighbors in $S_1$, while the other has one.}
         There are two possibilities. In the first case, without loss of generality, suppose $d$ is the neighbor of $a$ and $c$, while $e$ is the neighbor of $c$ (see Figure \ref{6-4}(vi)).  Similarly, as before, the vertices $a,c$ and $d$ cannot have a neighbor in $V(G)\backslash H$, whereas $e$ can have at most one neighbor in $V(G)\backslash H$.  If $d$ and $e$ are adjacent, then $e$ cannot have a neighbor in $V(G)\backslash H$.  If there are no edges between the vertices in $S_1$, the number of edges deleted is at most $12+1+1=14$.  In the other case, the vertices $a$ and $c$ can be adjacent without any constraints.  Thus, if there is an edge joining any two vertices in $S_1$, the number of edges deleted is at most $12+2+1=15$.

         Without loss of generality, suppose $d$ is the neighbor of $a$ and $c$, while $e$ is the neighbor of $b$ (see Figure \ref{6-4}(vii)).  Similarly, as before, the vertices $a,b,c$ and $d$ cannot have a neighbor in $V(G)\backslash H$, whereas $e$ can have at most one neighbor in $V(G)\backslash H$.  If $d$ and $e$ are adjacent, then $e$ cannot have a neighbor in $V(G)\backslash H$.  If there are no edges between the vertices in $S_1$, the number of edges deleted is at most $12+1=13$.  In the other case, the vertices $a,b$ and $c$ can be adjacent without any constraints.  Hence, the  number of edges deleted is at most $12+2+1=15$.
         
         \item \textbf{Both the vertices $d$ and $e$ have two neighbors in $S_1$.}
         Without loss of generality, suppose $d$ is the neighbor of $a$ and $c$, while $e$ is the neighbor of $b$ and $c$ (see Figure \ref{6-4}(viii)).  Similarly, as before, the vertices $a, b, c, d$ and $e$ cannot have a neighbor in $V(G)\backslash H$.  The vertex $c$ can be adjacent to $a$ and $b$ without any constraints.  Hence, the number of edges deleted is at most $13+2=15$.  
\end{enumerate}
Thus, by induction $$e(G)=e(G-H)+16\leq \frac{5}{2}(n-7)-2+16\leq \frac{5}{2}n-2,$$ and we are done.
\end{proof}

The following lemma completes the proof of the Theorem \ref{doublestarsmain}(v):
\begin{lemma}\label{cm5}
Let $G$ be an $S_{3,3}$-free plane graph on $n$ vertices, then $e(G)\leq \frac{5}{2}n-2.$ \end{lemma}
\begin{proof}
 Let  $A=\{x\in V(G)\ | \ d(x)=3\}$, $B=\{x\in V(G)\ | \ d(x)=4 \ \text{or} \ d(y)=5\}$ and   $C=\{x\in V(G)\ |\ d(x)\geq 6\}.$
 
 Since there is no $3$-$3$ edge, the vertices in $A$ are independent. From Claims \ref{clm5-6} and \ref{clm4-6}, there is no edge between the sets $B$ and $C$. Moreover, $C$ is independent by Claim \ref{clm6-6}. The distribution of the edges in $G$ is shown in Figure \ref{fg5}. 
 \begin{figure}[ht]
\centering
\captionsetup{justification=centering}
\begin{tikzpicture}[scale=0.125]
\draw[rotate around={90:(0,-7.5)}, thick,red] (0,-7.5) ellipse (23 and 9);
\draw[rotate around={90:(-30,-7.5)}, thick, red] (-30,-7.5) ellipse (23 and 9);
\draw[rotate around={90:(30,-7.5)},  thick,red] (30,-7.5) ellipse (23 and 9);
\draw[fill=black](0,10)circle(25pt);
\draw[fill=black](0,-10)circle(25pt);
\draw[fill=black](0,0)circle(25pt);
\draw[fill=black](0,-25)circle(25pt);
\draw[fill=black](0,-15)circle(5pt);
\draw[fill=black](0,-18)circle(5pt);
\draw[fill=black](0,-21)circle(5pt);
\draw[thick](0,10)--(10,12)(0,10)--(-10,8)(0,10)--(10,8);
\draw[thick](0,0)--(10,0)(0,0)--(10,2)(0,0)--(10,4);
\draw[thick](0,-10)--(-10,-10)(0,-10)--(-10,-8)(0,-10)--(-10,-12);
\draw[thick](0,-25)--(-10,-27)(0,-25)--(-10,-23)(0,-25)--(10,-25);
\draw[fill=black](30,10)circle(25pt);
\draw[fill=black](30,-10)circle(25pt);
\draw[fill=black](30,0)circle(25pt);
\draw[fill=black](30,-25)circle(25pt);
\draw[fill=black](30,-15)circle(5pt);
\draw[fill=black](30,-18)circle(5pt);
\draw[fill=black](30,-21)circle(5pt);
\draw[thick](30,10)--(20,18)(30,10)--(20,16)(30,10)--(20,14)(30,10)--(20,12)(30,10)--(20,8)(30,10)--(20,6)(30,10)--(20,10);
\draw[thick](30,0)--(20,4)(30,0)--(20,2)(30,0)--(20,0)(30,0)--(20,-2)(30,0)--(20,-4)(30,0)--(20,-6);
\draw[thick](30,-10)--(20,-8)(30,-10)--(20,-10)(30,-10)--(20,-12)(30,-10)--(20,-14)(30,-10)--(20,-16)(30,-10)--(20,-18)(30,-10)--(20,-20);
\draw[thick](30,-25)--(20,-22)(30,-25)--(20,-24)(30,-25)--(20,-26)(30,-25)--(20,-28)(30,-25)--(20,-30)(30,-25)--(20,-32);
\draw[fill=black](-30,10)circle(25pt);
\draw[fill=black](-30,-10)circle(25pt);
\draw[fill=black](-30,0)circle(25pt);
\draw[fill=black](-30,-25)circle(25pt);
\draw[fill=black](-30,-15)circle(5pt);
\draw[fill=black](-30,-18)circle(5pt);
\draw[fill=black](-30,-21)circle(5pt);
\draw[thick](-30,10)--(-20,10)(-30,10)--(-20,12)(-30,10)--(-20,8);
\draw[thick](-30,0)--(-20,2)(-30,0)--(-20,-2);
\draw[thick](-30,-10)--(-20,-8)(-30,-10)--(-20,-12);
\draw[thick](-30,-25)--(-20,-22)(-30,-25)--(-20,-24)(-30,-25)--(-20,-26);
\draw[thick](-30,10)..controls (-35, 7) and (-35,3) ..(-30,0);
\draw[thick](-30,10)..controls (-25, 5) and (-25,-5) ..(-30,-10);
\draw[thick](-30,-10)..controls (-25, -13) and (-25,-20) ..(-30,-25);
\draw[thick](-30,0)..controls (-35, -3) and (-35,-7) ..(-30,-10);
\node at (0,20) {$A$};
\node at (-30,20) {$B$};
\node at (30,20) {$C$};
\end{tikzpicture}
\caption{A graph showing the distribution of edges in $G$.}
\label{fg5}
\end{figure}
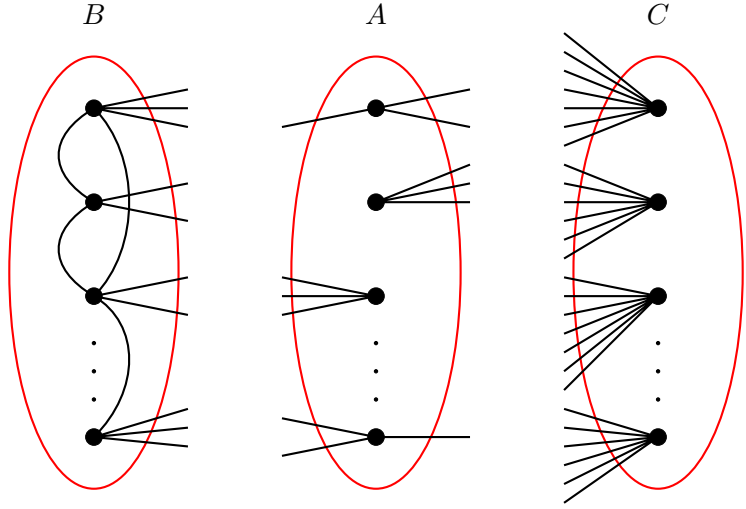

Let $|A|=a,$ $|B|=b$ and $|C|=c$. Let $x$ be the number of edges between the sets $A$ and $B$, i.e., $e(A,B)=x$. Since the maximum degree in $B$ is $5$, we have $2e(G[B])=\sum_{v\in B}d(v)-x$ which implies $e(G)\leq \frac{5b-x}{2}$. Each vertex in $A$ has degree $3$ and the vertices in $A$ are independent, hence $e(A,C)= 3a-x$. Thus, the number of edges in $G$ is 
\begin{align*}
e(G)=e(G[B])+e(A,B)+e(A,C)\leq \frac{5b-x}{2}+x+3a-x=\frac{5}{2}a+\frac{5}{2}b+\frac{a-x}{2}.
\end{align*}
On the other hand, since $G$ is a plane graph and the graph induced by the vertices in $A$ and the vertices in $C$ is bipartite, $e(A,C)=3a-x\leq 2(a+c)-4$. This implies that $\frac{a-x}{2}\leq c-2= \frac{5}{2}c-\frac{3}{2}c-2$ for all $c\geq 0$. Therefore, using the inequality in (\ref{eq1}), we get
\begin{align}\label{eq1}
e(G)\leq \frac{5}{2}a+\frac{5}{2}b+\frac{5}{2}c-\left(\frac{3}{2}c+2\right)=\frac{5}{2}(a+b+c)-\left(\frac{3}{2}c+2\right)\leq\frac{5}{2}n-2.    
\end{align}
The last inequality in~(\ref{eq1}) holds (and hence Lemma~\ref{cm5}) if $a\neq 0$ and $c\neq 0$. To finish the proof, we distinguish the following two cases:

\textbf{Case $1$: $a\neq 0$ and $c=0$.}
Observe that $e(G)\leq \frac{5(n-a)+3a}{2}=\frac{5}{2}n-a$. If $a\geq 2$, then we are done. Thus, $a=1$. Let the number of degree $4$ vertices in $G$ be $k$. Hence, $e(G)= \frac{5(n-k-1)+4k+3}{2}=\frac{5}{2}n-\frac{k}{2}-1$. If $k\geq 2$, then we are done. 

Let the number of degree $4$ vertices in $G$ be at most $1$. Let $A=\{x\}$ and $N(x)=\{x_1,x_2,x_3\}$. Let $d(x_i)=5$, for every $i\in\{1,2,3\}$. Considering that there is at most one degree $4$ vertex in $G$, it is easy to find a degree $5$ vertex in $\cup_{i=1}^{3}N(x_i)$, such that all its $5$ neighboring vertices are of degree $5$. Moreover, the same property holds if one vertex in $N(x)$ is of degree $4$. Let $v$ be a degree $5$ vertex in $G$, such that all its $5$ neighboring vertices are of degree $5$. Let $N(v)=\{x_1,x_2,x_3,x_4,x_5\}$, such that a plane drawing of $G$ results in a clockwise alignment of the vertices $x_1,x_2,x_3,x_4,x_5$ around $v$. Since $G$ is an $S_{3,3}$-free plane graph, every $5$-$5$ edge in $G$ must be contained in at least $3$ triangles. Thus, an edge $x_1v$ must be contained in at least $3$ triangles. This implies, $x_1$ must be adjacent to at least one vertex in $\{x_3,x_4\}$. Without loss of generality, assume $x_1$ and $x_3$ are adjacent. Then the $5$-$5$ edge $x_2v$ is contained in at most $2$ triangles, which results in an $S_{3,3}$ in $G$ with $x_2v$ as the backbone.   

\textbf{Case $2$: $a=0$.}
Let the number of degree $4$ vertices in $G$ be $k$. Thus, $e(G)=\frac{5(n-k)+4k}{2}=\frac{5}{2}n-\frac{k}{2}$. If the number of degree $4$ vertices is at least $4$, then $e(G)\leq \frac{5}{2}n-2$ and we are done. Now assume that the number of degree $4$ vertices in $G$ is at most $3$. Notice that, $v(G)\geq 8$. Otherwise, taking any maximal planar graph on $n$ vertices, it can be checked that $3n-6<\frac{5}{2}n-2$. 

Let $v$ be a degree $5$ vertex in $G$, and $N(v)=\{x_1,x_2,x_3,x_4,x_5\}$. At least two vertices in $N(v)$ must be of degree $5$. Otherwise, the number of degree $4$ vertices is at least $4$ and we are done. Let the plane drawing of $G$ result in a clockwise alignment of the vertices $x_1,x_2,x_3,x_4,x_5$ around $v$.  There are exactly $2$ vertices in $N(v)$, which are of degree $5$. Indeed, suppose that the number of degree $5$ vertices is at least $3$. We can assume that for some $i\in [5]$, $d(x_i)=d(x_{i+1})=5$. Without loss of generality, assume that these vertices are $x_1$ and $x_2$. Since $x_1v$ and $x_2v$ are $5$-$5$ edges, they must be contained in at least $3$ triangles. Thus, both $x_1$ and $x_2$ must be adjacent to $x_4$. On the other hand, it is easy to see that $x_3v$ and $x_5v$ are $4$-$5$ edges. Thus, the vertex $x_3$ must be adjacent to $x_2$ and $x_4$. Similarly, the vertex $x_5$ must be adjacent to $x_1$ and $x_4$.  Since $d(x_3)=4$, there must be a vertex $x_6$, such that $x_3x_6\in E(G)$. If $d(x_6)=5$, then $x_6$ must be adjacent to $x_4$. This is impossible, as $d(x_4)=5$. Hence, $d(x_6)=4$. Similarly, we have another vertex $x_7$ adjacent to $x_5$ and $d(x_7)=4$. This is a contradiction, as we found $4$ vertices of degree $4$, namely $x_3,x_5,x_6$ and $x_7$.

Thus, we can assume that only two vertices in $N(v)$ are of degree $5$. Moreover, the vertices are not consecutive with respect to the alignment in the clockwise direction. Without loss of generality, assume that the vertices are $x_2$ and $x_4$. It can be checked that the vertices $x_2$ and $x_4$ are adjacent. Since $x_1v$ and $x_5v$ are $4$-$5$ edges, then the edges $x_1x_2,x_1x_5$ and $x_5x_4$ are in $G$. Since $d(x_3)=4$, there must exist a vertex $x_6$ adjacent to the vertex $x_3$. If this vertex is of degree $4$, then it is a contradiction as we found $4$ vertices of degree $4$, namely $x_1,x_3,x_5$ and $x_6$. Hence, $d(x_6)=5$ and the edges $x_2x_6$ and $x_4x_6$ are in $G$. Since $d(x_1)$ is $4$, there must exist a vertex $x_7$ such that $x_1x_7\in E(G)$. If $d(x_7)$ is $5$, then $x_7$ is adjacent to $x_2$ and $d(x_2)\geq 6$, which is a contradiction. Hence, $x_7$ must be a vertex of degree $4$. This is a contradiction, as we found $4$ vertices of degree $4$, namely $x_1,x_3,x_5$ and $x_7$. This completes the proof of Claim \ref{cm5}, and subsequently the proof of Theorem \ref{doublestarsmain}(v). 
\end{proof}

\section{Planar Tur\'an number of \texorpdfstring{$S_{3,4}$}{S3,4}}
\begin{proof}[Proof of the Theorem \ref{doublestarsmain}(vi)] 
Let $G$ be an $n$-vertex $S_{3,4}$-free plane graph.  Since $S_{3,4}$ contains $9$ vertices, a maximal planar graph with $n\leq 8$ vertices, does not contain an $S_{3,4}$. Let $8|n$. Consider the plane graph consisting of $\frac{n}{8}$ disjoint copies of maximal planar graphs on $8$ vertices.  This graph does not contain an $S_{3,4}$. Hence, $\text{ex}_{\mathcal{P}}(n,S_{3,4})\geq\frac{9}{4}n$.  

\begin{claim}
Let $G$ be an $S_{3,4}$ on $n$ $(1\leq n\leq 42)$ vertices.  The number of edges in $G$ is at most $\frac{20}{7}n$.
\end{claim}
\begin{proof}
Recall that, an $n$-vertex maximal planar graph contains $3n-6$ edges. Since $3n-6\leq \frac{20}{7}n$ for $n\leq 42$,  $e(G)\leq \frac{20}{7}n$ holds for all $n$, $1\leq n\leq 42$.
\end{proof}

The following claims deal with the different cases of degree pairs in $G$:
\begin{claim}\label{S3,4_8}
No vertex in $G$ with a degree at least $8$ is adjacent to a vertex of degree at least $4$.
\end{claim}
\begin{proof}
 Suppose not. Let $xy$ be an edge in $G$ such that $d(x)\geq 8$ and $d(y)\geq 4$. Obviously, there are three vertices in $V(G)\backslash \{x\}$, say $y_1,y_2$ and $y_3$, which are adjacent to $y$. Since $|N(x)\backslash\{y\}|\geq 7$, there are four vertices $x_1,x_2,x_3$ and $x_4$, not in $\{y,y_1,y_2,y_3\}$ which are adjacent to $x$. This implies we got an $S_{3,4}$ in $G$ with backbone $xy$ and leaf-sets $\{x_1,x_2,x_3,x_4\}$ and $\{y_1,y_2,y_3\}$, respectively, which is a contradiction.   This completes the proof.
\end{proof}

\begin{claim}\label{S3,4_7,7}
If there is a $7$-$7$, $6$-$7$, $5$-$7$ and $6$-$6$ edge in $G$, then $e(G)\leq \frac{20}{7}n$.
\end{claim}
\begin{proof}
The proofs are similar to the Claims \ref{clm6-6}, \ref{clm5-6}, \ref{clm4-6}, and \ref{clm6-6s_2,5} respectively.  We skip it here for conciseness, but it is provided in the Appendix \ref{appendix}.
\end{proof}

Take $x,y\in V(G)$.  By the previous claims, if $d(x)+d(y)\geq 12$, we are done by induction.  Assume that $d(x)+d(y)\leq 11$.  Summing it over all the edge pairs in $G$, we have $11e\geq \sum_{xy\in E(G)}\left(d(x)+d(y)\right)=\sum_{x\in V(G)}(d(x))^2\geq n\overline{d}^2=n(\frac{2e}{n})^2$, where $\overline{d}$ is the average degree in $G$.  This gives us $e\leq \frac{11}{4}n\leq \frac{20}{7}n$ for $n\geq1$.  
\end{proof}

\section{Concluding remarks and Conclusions}

Concerning the exact value of $\ex_{\mathcal{P}}(n,S_{3,3})$, we conjecture the following:
\begin{conjecture}
\begin{equation*}\ex_{\p}(n,S_{3,3})=
\begin{cases}
3n-6, &\text{ if $3\leq n\leq 7$,}\\
16, &\text{if $n=8$,}\\
18, &\text{if $n=9$,}\\
\left\lfloor{\frac{5}{2}n}\right\rfloor-5, &\text{otherwise}.
\end{cases}
\end{equation*}
\end{conjecture}
\section{Acknowledgements}
Gy\H{o}ri's research was partially supported by the National Research, Development, and Innovation Office NKFIH, grants  K132696, SNN135643 and K126853. 

\printbibliography
\newpage
\appendix
\section{Proof of Lemma \ref{S3,4_7,7}}
\label{appendix}

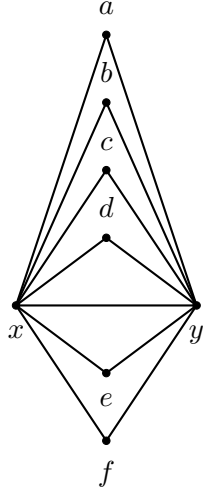
\begin{figure}[ht]
\centering
    \begin{tikzpicture}[scale=.1]
            \coordinate (x) at (-12,0);
            \coordinate (y) at (12,0);
            \coordinate (a) at (0,36);
            \coordinate (b) at (0,27);
            \coordinate (c) at (0,18);
            \coordinate (d) at (0,9);
            \coordinate (e) at (0,-9);
            \coordinate (f) at (0,-18);
            \draw[thick] (x) -- (a) -- (y) -- (x);
            \draw[thick] (x) -- (b) -- (y);
            \draw[thick] (x) -- (c) -- (y);
            \draw[thick] (x) -- (d) -- (y);
            \draw[thick] (x) -- (e) -- (y);
            \draw[thick] (x) -- (f) -- (y);
            \draw[fill=black] (x) circle(\circ)  node[label=below:$x$] {};
            \draw[fill=black] (y) circle(\circ)  node[label=below:$y$] {};
            \draw[fill=black] (a) circle(\circ)  node[label=above:$a$] {};
            \draw[fill=black] (b) circle(\circ)  node[label=above:$b$] {};
            \draw[fill=black] (c) circle(\circ)  node[label=above:$c$] {};
            \draw[fill=black] (d) circle(\circ)  node[label=above:$d$] {};
            \draw[fill=black] (e) circle(\circ)  node[label=below:$e$] {};
            \draw[fill=black] (f) circle(\circ)  node[label=below:$f$] {};
    \end{tikzpicture}
\caption{$G$ has a $7-7$ edge $xy$.}
\label{S3,4_7case}
\end{figure}

\begin{lemma}\label{S3,4_7,7_app}
If there is a $7$-$7$ edge in $G$, then $e(G)\leq \frac{20}{7}n$.
\end{lemma}
\begin{proof}
Let $xy\in E(G)$ be a $7$-$7$ edge.  Since $G$ is an $S_{3,4}$-free plane graph, $xy$ must be contained in $6$ triangles.  Let $a,b,c,d,e$ and $f$ be the vertices in $G$ which are adjacent to both $x$ and $y$, see Figure \ref{S3,4_7case}(i).  Let $S_1=\{a,b,c,d,e,f\}$ and $H=\{x,y,a,b,c,d,e,f\}$.   Delete the vertices in $H$.  Assume $a$ has two neighbors in $V(G)\backslash H$.  We immediately get an $S_{3,4}$ with $xa$ or $ya$ as the backbone.  Thus, any vertex in the set $S_1$ can have at most $1$ neighbor in $V(G)\backslash H$.  If there are no edges between the vertices in $S_1$, we deleted at most $13 + 6= 19$ edges. Assume that there is an edge between the vertices in $S_1$, say $ab$. If $a$ (or $b$) has a neighbor in $V(G)\backslash H$, then $xa$ (or $xb$) is the backbone of an $S_{3,3}$. Similarly, for the other edges in $S_1$, both the vertices cannot have a neighbor in $V(G)\backslash H$. Thus, if there is an edge joining any two vertices in $S_1$, the number of edges deleted is at most $13 + 5= 18$.  By the induction hypothesis, we get $e(G-H)\leq \frac{20}{7}(n-8)$. Hence, $e(G)=e(G-H)+19\leq\frac{20}{7}(n-8)+19\leq\frac{20}{7}n$.
\end{proof}
\begin{figure}[ht]
\centering
    \begin{tikzpicture}[scale=.1]
            \coordinate (x) at (-15,0);
            \coordinate (y) at (15,0);
            \coordinate (a) at (0,36);
            \coordinate (b) at (0,27);
            \coordinate (c) at (0,18);
            \coordinate (d) at (0,9);
            \coordinate (e) at (0,-9);
            \coordinate (f) at (24,0);
            \draw[thick] (x) -- (a) -- (y) -- (x);
            \draw[thick] (x) -- (b) -- (y);
            \draw[thick] (x) -- (c) -- (y);
            \draw[thick] (x) -- (d) -- (y);
            \draw[thick] (x) -- (e) -- (y);
            \draw[thick] (f) -- (y);
            \draw[fill=black] (x) circle(\circ)  node[label=below:$x$] {};
            \draw[fill=black] (y) circle(\circ)  node[label=below:$y$] {};
            \draw[fill=black] (a) circle(\circ)  node[label=below:$a$] {};
            \draw[fill=black] (b) circle(\circ)  node[label=below:$b$] {};
            \draw[fill=black] (c) circle(\circ)  node[label=below:$c$] {};
            \draw[fill=black] (d) circle(\circ)  node[label=below:$d$] {};
            \draw[fill=black] (e) circle(\circ)  node[label=above:$e$] {};
            \draw[fill=black] (f) circle(\circ)  node[label=below:$f$] {};
            \node at (0,-21){(i)};
    \end{tikzpicture}\qquad
    \begin{tikzpicture}[scale=.1]
            \coordinate (x) at (-15,0);
            \coordinate (y) at (15,0);
            \coordinate (a) at (0,36);
            \coordinate (b) at (0,27);
            \coordinate (c) at (0,18);
            \coordinate (d) at (0,9);
            \coordinate (e) at (0,-9);
            \coordinate (f) at (24,0);
            \draw[thick] (x) -- (a) -- (y) -- (x);
            \draw[thick] (x) -- (b) -- (y);
            \draw[thick] (x) -- (c) -- (y);
            \draw[thick] (x) -- (d) -- (y);
            \draw[thick] (x) -- (e) -- (y);
            \draw[thick] (f) -- (y);
            \draw[thick] (f) -- (a);
            \draw[fill=black] (x) circle(\circ)  node[label=below:$x$] {};
            \draw[fill=black] (y) circle(\circ)  node[label=below:$y$] {};
            \draw[fill=black] (a) circle(\circ)  node[label=below:$a$] {};
            \draw[fill=black] (b) circle(\circ)  node[label=below:$b$] {};
            \draw[fill=black] (c) circle(\circ)  node[label=below:$c$] {};
            \draw[fill=black] (d) circle(\circ)  node[label=below:$d$] {};
            \draw[fill=black] (e) circle(\circ)  node[label=above:$e$] {};
            \draw[fill=black] (f) circle(\circ)  node[label=below:$f$] {};
            \node at (0,-21){(ii)};
    \end{tikzpicture}\qquad
    \begin{tikzpicture}[scale=.1]
            \coordinate (x) at (-15,0);
            \coordinate (y) at (15,0);
            \coordinate (a) at (0,36);
            \coordinate (b) at (0,27);
            \coordinate (c) at (0,18);
            \coordinate (d) at (0,9);
            \coordinate (e) at (0,-9);
            \coordinate (f) at (24,0);
            \draw[thick] (x) -- (a) -- (y) -- (x);
            \draw[thick] (x) -- (b) -- (y);
            \draw[thick] (x) -- (c) -- (y);
            \draw[thick] (x) -- (d) -- (y);
            \draw[thick] (x) -- (e) -- (y);
            \draw[thick] (f) -- (y);
            \draw[thick] (f) -- (a);
            \draw[thick] (f) -- (e);
            \draw[fill=black] (x) circle(\circ)  node[label=below:$x$] {};
            \draw[fill=black] (y) circle(\circ)  node[label=above:$y$] {};
            \draw[fill=black] (a) circle(\circ)  node[label=below:$a$] {};
            \draw[fill=black] (b) circle(\circ)  node[label=below:$b$] {};
            \draw[fill=black] (c) circle(\circ)  node[label=below:$c$] {};
            \draw[fill=black] (d) circle(\circ)  node[label=below:$d$] {};
            \draw[fill=black] (e) circle(\circ)  node[label=above:$e$] {};
            \draw[fill=black] (f) circle(\circ)  node[label=below:$f$] {};
            \node at (0,-21){(iii)};
    \end{tikzpicture}
\captionsetup{singlelinecheck=off}
\caption[foo bar2]{$G$ has a $6-7$ edge $xy$.
\begin{enumerate}[label=(\roman*)]
    \item The vertex $f$ has no neighbors in $S_1$.
    \item The vertex $f$ has one neighbor in $S_1$.
    \item The vertex $f$ has two neighbors in $S_1$.
\end{enumerate}}
\label{7-6case}
\end{figure}
\begin{lemma}\label{S3,4_6,7}
If there is a $6$-$7$ edge in $G$, then $e(G)\leq \frac{20}{7}n$.
\end{lemma}
\begin{proof}
Let $xy\in E(G)$ be a $6$-$7$ edge.  Since $G$ is an $S_{3,4}$-free plane graph, $xy$ must be contained in $5$ triangles. Let $a,b,c,d$ and $e$ be the vertices in $G$ which are adjacent to both $x$ and $y$. Let $f$ be the vertex adjacent to $y$ but not to $x$, see Figure \ref{7-6case}. Let $S_1=\{a,b,c,d,e\}$ and $H=\{x,y,a,b,c,d,e,f\}$.  Delete the vertices in $H$.   Any vertex in $S_1$ can have at most $1$ neighbor in $V(G)\backslash H$.  If there is an edge joining any two vertices in $S_1$, say $ab$.  Similarly, as before, the vertices $a$ and $b$ cannot have a neighbor in $V(G)\backslash H$. We further distinguish the cases based on the number of edges from $f$ as follows:

\begin{enumerate}

\item \textbf{The vertex $f$ has no neighbors in $S_1$}, see Figure \ref{7-6case}(i). Clearly, the vertex $f$ can have at most $2$ neighbors in $V(G)\backslash H$. If there are no edges between the vertices in $S_1$, we deleted at most $12+2+5=19$ edges.  If there is an edge joining any two vertices in $S_1$, the number of edges deleted is at most $12+2+4=18$.

\item \textbf{The vertex $f$ has one neighbor in $S_1$.} Without loss of generality, suppose $a$ is the neighbor of $f$,  see Figure \ref{7-6case}(ii).  The vertex $f$ can have at most one neighbor in $V(G)\backslash H$.  If the vertex $a$ has a neighbor in $V(G)\backslash H$, we have an $S_{3,4}$ with $xa$ as the backbone.  If there are no edges between the vertices in $S_1$, we have deleted at most  $13+4+1=18$ edges.  If there is an edge joining any two vertices in $S_1$, the number of edges deleted is at most  $13+4+1=18$.

\item \textbf{The vertex $f$ has two neighbors in $S_1$.}  Without loss of generality, suppose $f$ is the neighbor of $a$ and $e$, see Figure \ref{7-6case}(iii). The vertex $f$ cannot have a neighbor in $V(G)\backslash H$, otherwise we have an $S_{3,4}$ with $yf$ as the backbone.  If either $a$ or $e$ has a neighbor in $V(G)\backslash H$, we have an $S_{3,4}$ with $xa$ or $xe$ as the backbone, respectively. Suppose there is no edge between the vertices in $S_1$.  The total number of edges deleted is at most  $14+3=17$. Suppose $a$ and $b$ are adjacent, then $b$ cannot have a neighbor in $V(G)\backslash H$.  Similarly, for the other edges in $S_1$ except $ad$, which is still possible without any extra restrictions.  If there is an edge joining any two vertices in $S_1$, the total number of edges deleted is at most  $14+4=18$. 

\end{enumerate}
Thus, by induction $$e(G)=e(G-H)+19\leq \frac{20}{7}(n-8)+19< \frac{20}{7}n,$$ and we are done.
\end{proof}

\begin{figure}
\centering
\begin{tikzpicture}[scale=0.1]
            \coordinate (x) at (-12,0);
            \coordinate (y) at (12,0);
            \coordinate (a) at (0,27);
            \coordinate (b) at (0,18);
            \coordinate (c) at (0,9);
            \coordinate (d) at (0,-9);
            \coordinate (e) at (21,6); 
            \coordinate (f) at (21,0); 
            \draw[thick] (x) -- (a) -- (y) -- (x);
            \draw[thick] (x) -- (b) -- (y);
            \draw[thick] (x) -- (c) -- (y);
            \draw[thick] (x) -- (d) -- (y);
            \draw[thick] (e) -- (y);
            \draw[thick] (f) -- (y);
            \draw[fill=black] (x) circle(\circ)  node[label=below:$x$] {};
            \draw[fill=black] (y) circle(\circ)  node[label=below:$y$] {};
            \draw[fill=black] (a) circle(\circ)  node[label=below:$a$] {};
            \draw[fill=black] (b) circle(\circ)  node[label=below:$b$] {};
            \draw[fill=black] (c) circle(\circ)  node[label=below:$c$] {};
            \draw[fill=black] (d) circle(\circ)  node[label=above:$d$] {};
            \draw[fill=black] (e) circle(\circ)  node[label=below:$e$] {};
            \draw[fill=black] (f) circle(\circ)  node[label=below:$f$] {};
            \node at (0,-18){(i)};
    \end{tikzpicture}\qquad
\begin{tikzpicture}[scale=0.1]
            \coordinate (x) at (-12,0);
            \coordinate (y) at (12,0);
            \coordinate (a) at (0,27);
            \coordinate (b) at (0,18);
            \coordinate (c) at (0,9);
            \coordinate (d) at (0,-9);
            \coordinate (e) at (21,6); 
            \coordinate (f) at (21,0); 
            \draw[thick] (x) -- (a) -- (y) -- (x);
            \draw[thick] (x) -- (b) -- (y);
            \draw[thick] (x) -- (c) -- (y);
            \draw[thick] (x) -- (d) -- (y);
            \draw[thick] (e) -- (y);
            \draw[thick] (e) -- (y);
            \draw[thick] (e) -- (a);
            \draw[thick] (f) -- (y);
            \draw[fill=black] (x) circle(\circ)  node[label=below:$x$] {};
            \draw[fill=black] (y) circle(\circ)  node[label=below:$y$] {};
            \draw[fill=black] (a) circle(\circ)  node[label=below:$a$] {};
            \draw[fill=black] (b) circle(\circ)  node[label=below:$b$] {};
            \draw[fill=black] (c) circle(\circ)  node[label=below:$c$] {};
            \draw[fill=black] (d) circle(\circ)  node[label=above:$d$] {};
            \draw[fill=black] (e) circle(\circ)  node[label=right:$e$] {};
            \draw[fill=black] (f) circle(\circ)  node[label=below:$f$] {};
            \node at (0,-18){(ii)};
    \end{tikzpicture}
\qquad
\begin{tikzpicture}[scale=0.1]
           \coordinate (x) at (-12,0);
            \coordinate (y) at (12,0);
            \coordinate (a) at (0,27);
            \coordinate (b) at (0,18);
            \coordinate (c) at (0,9);
            \coordinate (d) at (0,-9);
            \coordinate (e) at (21,6); 
            \coordinate (f) at (21,0);  
            \draw[thick] (x) -- (a) -- (y) -- (x);
            \draw[thick] (x) -- (b) -- (y);
            \draw[thick] (x) -- (c) -- (y);
            \draw[thick] (x) -- (d) -- (y);
            \draw[thick] (e) -- (y);
            \draw[thick] (e) -- (a);
            \draw[thick] (f) -- (y);
            \draw[thick](f)..controls (27,12) and (27,18) ..(a);
            \draw[fill=black] (x) circle(\circ)  node[label=below:$x$] {};
            \draw[fill=black] (y) circle(\circ)  node[label=below:$y$] {};
            \draw[fill=black] (a) circle(\circ)  node[label=below:$a$] {};
            \draw[fill=black] (b) circle(\circ)  node[label=below:$b$] {};
            \draw[fill=black] (c) circle(\circ)  node[label=below:$c$] {};
            \draw[fill=black] (d) circle(\circ)  node[label=above:$d$] {};
            \draw[fill=black] (e) circle(\circ)  node[label=left:$e$] {};
            \draw[fill=black] (f) circle(\circ)  node[label=below:$f$] {};
            \node at (0,-18){(iii)};
    \end{tikzpicture}
\qquad
\begin{tikzpicture}[scale=0.1]
            \coordinate (x) at (-12,0);
            \coordinate (y) at (12,0);
            \coordinate (a) at (0,27);
            \coordinate (b) at (0,18);
            \coordinate (c) at (0,9);
            \coordinate (d) at (0,-9);
            \coordinate (e) at (21,6); 
            \coordinate (f) at (21,0);  
            \draw[thick] (x) -- (a) -- (y) -- (x);
            \draw[thick] (x) -- (b) -- (y);
            \draw[thick] (x) -- (c) -- (y);
            \draw[thick] (x) -- (d) -- (y);
            \draw[thick] (e) -- (y);
            \draw[thick] (e) -- (a);
            \draw[thick] (f) -- (y);
            \draw[thick](f)--(d);
            \draw[fill=black] (x) circle(\circ)  node[label=above:$x$] {};
            \draw[fill=black] (y) circle(\circ)  node[label=above:$y$] {};
            \draw[fill=black] (a) circle(\circ)  node[label=below:$a$] {};
            \draw[fill=black] (b) circle(\circ)  node[label=below:$b$] {};
            \draw[fill=black] (c) circle(\circ)  node[label=below:$c$] {};;
            \draw[fill=black] (d) circle(\circ)  node[label=above:$d$] {};
            \draw[fill=black] (e) circle(\circ)  node[label=below:$e$] {};
            \draw[fill=black] (f) circle(\circ)  node[label=below:$f$] {};
            \node at (0,-18){(iv)};
    \end{tikzpicture}
\qquad
\begin{tikzpicture}[scale=0.1]
            \coordinate (x) at (-12,0);
            \coordinate (y) at (12,0);
            \coordinate (a) at (0,27);
            \coordinate (b) at (0,18);
            \coordinate (c) at (0,9);
            \coordinate (d) at (0,-9);
            \coordinate (e) at (21,6); 
            \coordinate (f) at (21,0); 
            \draw[thick] (x) -- (a) -- (y) -- (x);
            \draw[thick] (x) -- (b) -- (y);
            \draw[thick] (x) -- (c) -- (y);
            \draw[thick] (x) -- (d) -- (y);
            \draw[thick] (e) -- (y);
            \draw[thick] (e) -- (a);
            \draw[thick] (f) -- (y);
            \draw[thick](e)..controls (27,0) and (27,-12) ..(d);
            \draw[fill=black] (x) circle(\circ)  node[label=above:$x$] {};
            \draw[fill=black] (y) circle(\circ)  node[label=above:$y$] {};
            \draw[fill=black] (a) circle(\circ)  node[label=below:$a$] {};
            \draw[fill=black] (b) circle(\circ)  node[label=below:$b$] {};
            \draw[fill=black] (c) circle(\circ)  node[label=below:$c$] {};
            \draw[fill=black] (d) circle(\circ)  node[label=above:$d$] {};
            \draw[fill=black] (e) circle(\circ)  node[label=below:$e$] {};
            \draw[fill=black] (f) circle(\circ)  node[label=below:$f$] {};
            \node at (0,-18){(v)};
    \end{tikzpicture}
\qquad
\begin{tikzpicture}[scale=0.1]
            \coordinate (x) at (-12,0);
            \coordinate (y) at (12,0);
            \coordinate (a) at (0,27);
            \coordinate (b) at (0,18);
            \coordinate (c) at (0,9);
            \coordinate (d) at (0,-9);
            \coordinate (e) at (21,6); 
            \coordinate (f) at (21,0); 
            \draw[thick] (x) -- (a) -- (y) -- (x);
            \draw[thick] (x) -- (b) -- (y);
            \draw[thick] (x) -- (c) -- (y);
            \draw[thick] (x) -- (d) -- (y);
            \draw[thick] (e) -- (y);
            \draw[thick] (e) -- (a);
            \draw[thick] (f) -- (y);
            \draw[thick] (f) -- (d);
            \draw[thick](e)..controls (27,0) and (27,-12) ..(d);
            \draw[fill=black] (x) circle(\circ)  node[label=above:$x$] {};
            \draw[fill=black] (y) circle(\circ)  node[label=above:$y$] {};
            \draw[fill=black] (a) circle(\circ)  node[label=below:$a$] {};
            \draw[fill=black] (b) circle(\circ)  node[label=below:$b$] {};
            \draw[fill=black] (c) circle(\circ)  node[label=below:$c$] {};
            \draw[fill=black] (d) circle(\circ)  node[label=above:$d$] {};
            \draw[fill=black] (e) circle(\circ)  node[label=below:$e$] {};
            \draw[fill=black] (f) circle(\circ)  node[label=below:$f$] {};
            \node at (0,-18){(vi)};
    \end{tikzpicture}
\qquad
\begin{tikzpicture}[scale=0.1]
            \coordinate (x) at (-12,0);
            \coordinate (y) at (12,0);
            \coordinate (a) at (0,18);
            \coordinate (b) at (0,9);
            \coordinate (c) at (0,-3);
            \coordinate (d) at (0,-36);
            \coordinate (e) at (21,6); 
            \coordinate (f) at (3,-6); 
            \draw[thick] (x) -- (a) -- (y) -- (x);
            \draw[thick] (x) -- (b) -- (y);
            \draw[thick] (x) -- (c) -- (y);
            \draw[thick] (x) -- (d) -- (y);
            \draw[thick] (e) -- (y);
            \draw[thick] (e) -- (a);
            \draw[thick] (f) -- (y);
            \draw[thick] (f) -- (c);
            \draw[thick](e)..controls (27,0) and (27,-12) ..(d);
            \draw[fill=black] (x) circle(\circ)  node[label=above:$x$] {};
            \draw[fill=black] (y) circle(\circ)  node[label=above:$y$] {};
            \draw[fill=black] (a) circle(\circ)  node[label=below:$a$] {};
            \draw[fill=black] (b) circle(\circ)  node[label=below:$b$] {};
            \draw[fill=black] (c) circle(\circ)  node[label=below:$c$] {};
            \draw[fill=black] (d) circle(\circ)  node[label=below:$d$] {};
            \draw[fill=black] (e) circle(\circ)  node[label=right:$e$] {};
            \draw[fill=black] (f) circle(\circ)  node[label=right:$f$] {};
            \node at (0,-45){(vii)};
    \end{tikzpicture}\qquad
\begin{tikzpicture}[scale=0.1]
            \coordinate (x) at (-12,0);
            \coordinate (y) at (12,0);
            \coordinate (a) at (0,18);
            \coordinate (b) at (0,9);
            \coordinate (c) at (0,-3);
            \coordinate (d) at (0,-36);
            \coordinate (e) at (21,6); 
            \coordinate (f) at (3,-6); 
            \draw[thick] (x) -- (a) -- (y) -- (x);
            \draw[thick] (x) -- (b) -- (y);
            \draw[thick] (x) -- (c) -- (y);
            \draw[thick] (x) -- (d) -- (y);
            \draw[thick] (e) -- (y);
            \draw[thick] (e) -- (a);
            \draw[thick] (f) -- (y);
            \draw[thick] (f) -- (c);
            \draw[thick] (f) -- (d);
            \draw[thick](e)..controls (27,0) and (27,-12) ..(d);
            \draw[fill=black] (x) circle(\circ)  node[label=above:$x$] {};
            \draw[fill=black] (y) circle(\circ)  node[label=above:$y$] {};
            \draw[fill=black] (a) circle(\circ)  node[label=below:$a$] {};
            \draw[fill=black] (b) circle(\circ)  node[label=below:$b$] {};
            \draw[fill=black] (c) circle(\circ)  node[label=below:$c$] {};
            \draw[fill=black] (d) circle(\circ)  node[label=below:$d$] {};
            \draw[fill=black] (e) circle(\circ)  node[label=right:$e$] {};
            \draw[fill=black] (f) circle(\circ)  node[label=right:$f$] {};
            \node at (0,-45){(viii)};
    \end{tikzpicture}\qquad
    \begin{tikzpicture}[scale=0.1]
            \coordinate (x) at (-12,0);
            \coordinate (y) at (12,0);
            \coordinate (a) at (0,9);
            \coordinate (b) at (0,-3);
            \coordinate (c) at (0,-27);
            \coordinate (d) at (0,-36);
            \coordinate (e) at (21,6); 
            \coordinate (f) at (3,-6); 
            \draw[thick] (x) -- (a) -- (y) -- (x);
            \draw[thick] (x) -- (b) -- (y);
            \draw[thick] (x) -- (c) -- (y);
            \draw[thick] (x) -- (d) -- (y);
            \draw[thick] (e) -- (y);
            \draw[thick] (e) -- (a);
            \draw[thick] (f) -- (y);
            \draw[thick] (f) -- (c);
            \draw[thick] (f) -- (b);
            \draw[thick](e)..controls (27,0) and (27,-12) ..(d);
            \draw[fill=black] (x) circle(\circ)  node[label=above:$x$] {};
            \draw[fill=black] (y) circle(\circ)  node[label=above:$y$] {};
            \draw[fill=black] (a) circle(\circ)  node[label=below:$a$] {};
            \draw[fill=black] (b) circle(\circ)  node[label=below:$b$] {};
            \draw[fill=black] (c) circle(\circ)  node[label=below:$c$] {};
            \draw[fill=black] (d) circle(\circ)  node[label=below:$d$] {};
            \draw[fill=black] (e) circle(\circ)  node[label=right:$e$] {};
            \draw[fill=black] (f) circle(\circ)  node[label=right:$f$] {};
            \node at (0,-45){(ix)};
    \end{tikzpicture}
\end{figure}
\begin{figure}
\centering
\captionsetup{singlelinecheck=off}
\caption[foo bar]{$G$ has a $5-7$ edge $xy$.
\begin{enumerate}[label=(\roman*)]
    \item The vertices $e$ and $f$ have no neighbors in $S_1$.
    \item The vertex $e$ has one neighbor in $S_1$, and $f$ has none.
    \item Both the vertices $e$ and $f$ have one common neighbor in $S_1$.
    \item The vertices $e$ and $f$ have one neighbor in $S_1$ and they are distinct.
    \item One of the vertices $e$ or $f$ has two neighbors in $S_1$.
    \item Suppose $e$ is neighbor of $a$ and $d$, while $f$ is neighbor of $d$.
    \item Suppose $e$ is neighbor of $a$ and $d$, while $f$ is neighbor of $c$.
    \item Suppose $e$ is neighbor of $a$ and $d$, while $f$ is neighbor of $c$ and $d$.
    \item Suppose $e$ is neighbor of $a$ and $d$, while $f$ is neighbor of $b$ and $c$.
\end{enumerate}}
\ContinuedFloat
\label{7-5}
\end{figure}

\begin{lemma}\label{S3,4_5,7}
If there is a $5$-$7$ edge in $G$, then $e(G)\leq \frac{20}{7}n$.
\end{lemma}
\begin{proof}
Let $xy$ be a $5-7$ edge in $G$.  Since $G$ is an $S_{3,4}$-free plane graph, $xy$ must be contained in $4$ triangles.  Let $a,b,c$ and $d$ be the vertices in $G$ which are adjacent to both $x$ and $y$.  Let $e$ and $f$ be the vertices adjacent to $y$ but not to $x$, see Figure \ref{7-5}.   Let $S_1=\{a,b,c,d\}$ and $H=\{x,y,a,b,c,d,e,f\}$.  Delete the vertices in $H$.  

The vertices $e$ and $f$ can have at most two neighbors in $V(G)\backslash H$ each.  The vertices in $S_1$ can have at most one neighbor in $V(G)\backslash H$ each.  If there is an edge joining any two vertices in $S_1$, say $ab$.  Similarly, as before, the vertices $a$ and $b$ cannot have a neighbor in $V(G)\backslash H$.   We distinguish the cases based on the neighbors of $e$ and $f$ as follows:
\begin{enumerate}
         
         \item \textbf{The vertices $e$ and $f$ have no neighbors in $S_1$}, see Figure \ref{7-5}(i)).
         If the vertices $e$ and $f$ are adjacent, they can have at most one neighbor in $V(G)\backslash H$ each.  Otherwise, we have an $S_{3,4}$ with $ey$ (or $fy$) as the backbone.  If there are no edges between the vertices in $S_1$, the number of edges deleted is at most $11+4+4=19$.  If there is an edge joining any two vertices in $S_1$, the number of edges deleted is at most $11+3+4=18$.

         \item \textbf{One of the vertices $e$ or $f$ has one neighbor in $S_1$, while the other has none.}
         Without loss of generality, suppose $e$ and $a$ are adjacent (see Figure \ref{7-5}(ii)).  The vertex $a$ cannot have a neighbor in $V(G)\backslash H$, and $e$ can have at most one neighbor in $V(G)\backslash H$.  If the vertices $e$ and $f$ are adjacent, then $e$ cannot have a neighbor in $V(G)\backslash H$ and $f$ can have at most one neighbor in $V(G)\backslash H$.  If there are no edges between the vertices in $S_1$, the number of edges deleted is at most $12+3+3=18$.  If there is an edge joining any two vertices in $S_1$, the number of edges deleted is at most $12+3+3=18$.

         \item \textbf{Both the vertices $e$ and $f$ have one neighbor in $S_1$.}
          There are two possibilities. In the first case, without loss of generality, suppose $a$ is the common neighbor of $e$ and $f$ (see Figure \ref{7-5}(iii)).  Similarly, as before, the vertex $a$ cannot have a neighbor in $V(G)\backslash H$, and $e$ and $f$ can have at most one neighbor in $V(G)\backslash H$ each.  If the vertices $e$ and $f$ are adjacent, $e$ and $f$ have no neighbors in $V(G)\backslash H$.  If there are no edges between the vertices in $S_1$, the number of edges deleted is at most $13+3+2=18$.  If there is an edge joining any two vertices in $S_1$, the number of edges deleted is at most $13+3+2=18$.  
          
          Without loss of generality, suppose $a$ is the neighbor of $e$ and $d$ is the neighbor of $f$ (see Figure \ref{7-5}(iv)). Similarly, as before, the vertices $a$ and $d$ cannot have a neighbor in $V(G)\backslash H$, and $e$ and $f$ can have at most one neighbor in $V(G)\backslash H$ each.  If the vertices $e$ and $f$ are adjacent, then they have no neighbor in $V(G)\backslash H$.  If there are no edges between the vertices in $S_1$, the number of edges deleted is at most $13+2+2=17$.   If $a$ and $b$ are adjacent, then $b$ does not have a neighbor in $V(G)\backslash H$.  Similarly, for the other edges in $S_1$, except $ad$.  The vertices $a$ and $d$ can be adjacent without any constraints.  Thus, the number of edges deleted is at most $13+3+2=18$.

         \item \textbf{One of the vertices $e$ or $f$ has two neighbors in $S_1$, while the other has none.}  Without loss of generality, suppose $e$ is the neighbor of $a$ and $d$, while $f$ has no neighbors in $S_1$ (see Figure \ref{7-5}(v)).  The vertex $e$ cannot have a neighbor in $V(G)\backslash H$.  Similarly, as before, the vertices $a$ and $d$ cannot have a neighbor in $V(G)\backslash H$, whereas $f$ can have at most two neighbors in $V(G)\backslash H$.  If $e$ and $f$ are adjacent, then $f$ can have at most one neighbor in $V(G)\backslash H$.  If there are no edges between the vertices in $S_1$, the number of edges deleted is at most $13+2+2=17$.  If $a$ and $b$ are adjacent, then $b$ does not have a neighbor in $V(G)\backslash H$.  Similarly, for the other edges in $S_1$, except $ad$.  The vertices $a$ and $d$ can be adjacent without any constraints.  Hence, the number of edges deleted is at most $13+3+2=18$.

         \item \textbf{One of the vertices $e$ or $f$ has two neighbors in $S_1$, while the other has one.}
         There are two possibilities. In the first case, without loss of generality, suppose $e$ is the neighbor of $a$ and $d$, while $f$ is the neighbor of $d$ (see Figure \ref{7-5}(vi)).  Similarly, as before, the vertices $a,d$ and $e$ cannot have a neighbor in $V(G)\backslash H$, whereas $f$ can have at most one neighbor in $V(G)\backslash H$.  If $e$ and $f$ are adjacent, then $f$ cannot have a neighbor in $V(G)\backslash H$.  If there are no edges between the vertices in $S_1$, the number of edges deleted is at most $14+2+1=17$.  In the other case, the vertices $a$ and $d$ can be adjacent without any constraints.  Hence, the  number of edges deleted is at most $14+3+1=18$. 
         
         Without loss of generality, suppose $e$ is the neighbor of $a$ and $d$, while $f$ is the neighbor of $c$ (see Figure \ref{7-5}(vii)).  Similarly, as before, the vertices $a,c, d$ and $e$ cannot have a neighbor in $V(G)\backslash H$, whereas $f$ can have at most one neighbor in $V(G)\backslash H$.  If there are no edges between the vertices in $S_1$, the  number of edges deleted is at most $14+1+1=16$.  In the other case, the vertices $a,c$ and $d$ can be adjacent without any constraints.  Hence, the  number of edges deleted is at most $14+3+1=18$.

        \item \textbf{Both the vertices $e$ and $f$ have two neighbors in $S_1$.}
         There are two possibilities. In the first case, without loss of generality, suppose $e$ is the neighbor of $a$ and $d$, while $f$ is the neighbor of $c$ and $d$ (see Figure \ref{7-5}(viii)).  Similarly, as before, the vertices $a, c, d, e$ and $f$ cannot have a neighbor in $V(G)\backslash H$.  If there are no edges between the vertices in $S_1$, the  number of edges deleted is at most $15+1=16$.   In the other case, the vertices $a,c$ and $d$ can be adjacent without any constraints.  Thus, the number of edges deleted is $15+3=18$.  
         
         Without loss of generality, suppose $e$ is the neighbor of $a$ and $d$, while $f$ is the neighbor of $b$ and $c$ (see Figure \ref{7-5}(ix)).  Similarly, as before, the vertices $a, b, c, d, e$ and $f$ cannot have a neighbor in $V(G)\backslash H$.  The vertices $a,b,c$ and $d$ can be a path of length $3$ without any constraints.  In this case, the  number of edges deleted is $15+3=18$.
\end{enumerate}
Thus, by induction $$e(G)=e(G-H)+19\leq \frac{20}{7}(n-8)+19< \frac{20}{7}n,$$ and we are done.
\end{proof}

\begin{lemma}\label{S3,4_6,6}
If there is a $6$-$6$ edge in $G$, then $e(G)\leq \frac{20}{7}n$.
\end{lemma}

\begin{proof}
Let $xy\in E(G)$ be a $6$-$6$ edge. There are at least $4$ triangles sitting on the edge $xy$, otherwise $G$ contains an $S_{3,4}$.  We subdivide the cases based on the number of triangles sitting on the edge $xy$.
 \begin{enumerate}
     \item \textbf{There are $5$ triangles sitting on the edge $xy$.}   Let $a,b,c,d$ and $e$ be the vertices in $G$ which are adjacent to both $x$ and $y$, see Figure \ref{fig_s_3,4_6,6,5triangles}.  Let $S_1=\{a,b,c,d,e\}$, and $H=S_1\cup \{x,y\}$.  Delete the vertices in $H$.  The vertices in $S_1$ can have at most one neighbor in $V(G)\backslash H$ each and can form a path of length $4$ within $S_1$.  Hence, the number of edges deleted is $11+5+4=20$.  By the induction hypothesis, we get $e(G-H)\leq \frac{20}{7}(n-7)$. Hence, $e(G)=e(G-H)+20\leq \frac{20}{7}(n-7)+20\leq\frac{20}{7}n$.
     
      \item \textbf{There are $4$ triangles sitting on the edge $xy$.}  Let $a,b,c$ and $d$ be the vertices in $G$ which are adjacent to both $x$ and $y$.  Let $e$ be the vertex adjacent to $x$ but not adjacent to $y$, and $f$ be adjacent to $y$ but not adjacent to $x$.   Let $S_1=\{a,b,c,d\}$ and $H=\{x,y\}\cup S_1\cup\{e,f\}$.  Delete the vertices in $H$.  The vertices $e$ and $f$ can have at most two neighbors in $V(G)\backslash H$ each.  We distinguish the cases based on the neighbors of $e$ and $f$ as follows:
     \begin{enumerate}
         \item \textbf{The vertices $e$ and $f$ have no neighbors in $S_1$}, see Figure \ref{fig_s_3,4_6,6}(i)).
         The vertices in $S_1$ can have at most one neighbor in $V(G)\backslash H$ each and can form a path of length $3$ in $S_1$.  If the vertices $e$ and $f$ are adjacent, then both can have at most one neighbor in $V(G)\backslash H$.  Otherwise, we have an $S_{3,4}$ with $ex$ (or $fy$) as the backbone.  Thus, the number of edges deleted is at most $11+3+4+4=22$.  
         
         \item \textbf{One of the vertices $e$ or $f$ has one neighbor in $S_1$, while the other has none.} 
         Without loss of generality, assume that $e$ and $a$ are adjacent, see Figure \ref{fig_s_3,4_6,6}(ii).  The vertex $a$ cannot have a neighbor in $V(G)\backslash H$, and $e$ can have at most one neighbor in $V(G)\backslash H$.  Otherwise, if $a$ has a neighbor in $V(G)\backslash H$, $ya$ is the backbone of an $S_{3,4}$.  If the vertices $e$ and $f$ are adjacent, then the vertex $e$ cannot have a neighbor in $V(G)\backslash H$, and $f$ can have at most one neighbor in $V(G)\backslash H$.  Otherwise, we have an $S_{3,4}$ with $ex$ (or $fy$) as the backbone.   Similarly, as before, the vertices $b,c$ and $d$ can have at most one neighbor in $V(G)\backslash H$ each and the vertices $\{a,b,c,d\}$ can form a path of length $3$ in $S_1$.  Thus, the  number of edges deleted is at most $12+3+3+3=21$.  
         
         \item \textbf{The vertices $e$ and $f$ have one neighbor in $S_1$.}
         There are two possibilities. In the first case, without loss of generality, suppose $a$ is the common neighbor of $e$ and $f$, see Figure \ref{fig_s_3,4_6,6}(iii).  Similarly, as before, $a$ cannot have a neighbor in $V(G)\backslash H$, and $e$ and $f$ can have at most one neighbor in $V(G)\backslash H$ each.  The vertices $b,c$ and $d$ can  have  at most one neighbor in $V(G)\backslash H$ each,  and  the  vertices $\{a,b,c,d\}$ can form  a path of length $3$ in $S_1$.  If the vertices $e$ and $f$ are adjacent, then they cannot have a neighbor in $V(G)\backslash H$.  Thus, the number of edges deleted is at most $13+3+3+2=21$. 
         
         Without loss of generality, suppose $a$ is the neighbor of $e$, and $d$ is the neighbor of $f$, see Figure \ref{fig_s_3,4_6,6}(iv). Similarly, as before, the vertices $a$ and $d$ cannot have a neighbor in $V(G)\backslash H$, and $e$ and $f$ can have at most one neighbor in $V(G)\backslash H$ each.  The  vertices $b$ and $c$ can  have  at most one neighbor in $V(G)\backslash H$ each, and  the  vertices $\{a,b,c,d\}$ can form a path of length $3$ in $S_1$.  If the vertices $e$ and $f$ are adjacent, then they cannot have a neighbor in $V(G)\backslash H$.  Thus, the  number of edges deleted is at most $13+3+2+2=20$. (In fact, it can be shown that if the vertices $e$ and $f$ are adjacent, there can only be a $2$-path inside $S_1$. However, this precision is unnecessary. We skip this in the following cases also.)
         
        \item \textbf{One of the vertices $e$ or $f$ has two neighbors in $S_1$, while the other has none.}
         Without loss of generality, suppose $e$ is the neighbor of $a$ and $d$, see Figure \ref{fig_s_3,4_6,6}(v). Similarly, as before, $a$ and $d$ cannot have a neighbor in $V(G)\backslash H$, and $e$ at most one neighbor in $V(G)\backslash H$.  The vertices $b$ and $c$ can have at most one neighbor in $V(G)\backslash H$ each, and the vertices $\{a,b,c,d\}$ can form a path of length $3$ in $S_1$.  If the vertices $e$ and $f$ are adjacent, then $e$ cannot have a neighbor in $V(G)\backslash H$ and $f$ can have at most one neighbor in $V(G)\backslash H$.  Thus, the number of edges deleted is at most $13+3+2+3=21$.
         
         \item \textbf{One of the vertices $e$ or $f$ has two neighbors in $S_1$, while the other has one neighbor.}
         There are two possibilities. In the first case, without loss of generality, suppose $e$ is the neighbor of $a$ and $d$, and $f$ is the neighbor of $d$ (see Figure \ref{fig_s_3,4_6,6}(vi)).  Similarly, as before, the vertices $a$ and $d$ cannot have a neighbor in $V(G)\backslash H$.  The vertices $e$ and $f$ can have at most one neighbor in $V(G)\backslash H$ each.  On the other hand, the vertices  $b$ and $c$ can have at most one neighbor in $V(G)\backslash H$ each and the vertices $\{a,b,c,d\}$ can form a path of length $3$ in $S_1$.  If the vertices $e$ and $f$ are adjacent, $e$ and $f$ cannot have a neighbor in $V(G)\backslash H$. Thus, the number of edges deleted is at most $14+3+2+2=21$.
         
         Without loss of generality, assume that $e$ is the neighbor of $a$ and $d$, and $f$ is the neighbor of $b$ (see Figure \ref{fig_s_3,4_6,6}(vii)).  The vertices $a,b$ and $d$ cannot have a neighbor in $V(G)\backslash H$, but they along with $c$ can form a path of length $3$ in $S_1$.  The vertices $c,e$ and $f$ can have at most one neighbor in $V(G)\backslash H$ each.  Thus, the  number of edges deleted is at most $14+3+1+2=20$.
         
         \item \textbf{Both the vertices $e$ and $f$ have two neighbors in $S_1$.}  There are three possibilities. In the first case, without loss of generality, suppose $e$ and $f$ are neighbors of $a$ and $d$ both, see Figure \ref{fig_s_3,4_6,6}(viii).  Similarly, as before, the vertices $a$ and $d$ cannot have a neighbor in $V(G)\backslash H$.  The vertices $e$ and $f$ can have at most one neighbor in $V(G)\backslash H$ each.  The vertices $b$ and $c$ can have at most one neighbor in $V(G)\backslash H$ each.  The vertices $\{a,b,c,d\}$ can form a path of length $3$ in $S_1$.  If the vertices $e$ and $f$ are adjacent, $e$ and $f$ cannot have a neighbor in $V(G)\backslash H$.  Thus, the number of edges deleted is at most $15+3+2+2=22$.
         
         On the other hand, without loss of generality, assume $e$ is the neighbor of $a$ and $d$ both, while $f$ is the neighbor of $a$ and $b$ (see Figure \ref{fig_s_3,4_6,6}(ix)).  The vertices $a,b$ and $d$ cannot have a neighbor in $V(G)\backslash H$, but they along with $c$ can form a path of length $3$ in $S_1$. The vertices $c,e$ and $f$ can have at most one neighbor in $V(G)\backslash H$ each.  Thus, the  number of edges deleted is at most $15+3+1+2=21$.  
         
         In the last case, without loss of generality, assume $e$ is the neighbor of $a$ and $d$ both, while $f$ is the neighbor of $b$ and $c$ (see Figure \ref{fig_s_3,4_6,6}(x)).  The vertices $a,b,c$ and $d$ cannot have a neighbor in $V(G)\backslash H$, but they can form a path of length $3$ in $S_1$.  The vertices $e$ and $f$ can have at most one neighbor in $V(G)\backslash H$ each.  Thus, the number of edges deleted is at most $15+3+2=20$.
     \end{enumerate}
     By the induction hypothesis, we get $e(G-H)\leq \frac{20}{7}(n-8)$. Hence, $e(G)=e(G-H)+22\leq \frac{20}{7}(n-8)+22\leq\frac{20}{7}n$.
 \end{enumerate}
 
\end{proof}
\end{document}